\documentclass[11pt]{amsart}
\usepackage[latin1]{inputenc}
\usepackage{amssymb}
\usepackage[english]{babel}

\usepackage[pdftex,pagebackref,colorlinks=true,urlcolor=blue,linkcolor=blue,citecolor=blue]{hyperref}
\usepackage{graphicx}

\usepackage[capitalize]{cleveref}
\usepackage{color}
\usepackage{amsmath}
\usepackage{amsfonts}
\usepackage{mathrsfs}
\usepackage{t1enc , graphicx}
\usepackage{verbatim}
\usepackage{bbm}

\newcommand{\R}{\mathbb{R}}
\newcommand{\C}{\mathbb{C}}
\newcommand{\T}{\mathbb{T}}
\newcommand{\Q}{\mathbb{Q}}
\newcommand{\Z}{\mathbb{Z}}
\newcommand{\N}{\mathbb{N}}
\newcommand{\one}{\mathbbm{1}}

\newtheorem{theorem}{Theorem}[section]
\newtheorem{proposition}[theorem]{Proposition}

\newtheorem{lemma}[theorem]{Lemma}
\newtheorem{claim}[theorem]{Claim}

\theoremstyle{definition}
\newtheorem{definition}[theorem]{Definition}

\theoremstyle{remark}

\newtheorem{example}[theorem]{Example}

\newtheorem{remark}[theorem]{Remark}
\newtheorem*{remark*}{Remark}

\def \a {\alpha}
\def \i {\mathbf{i}}
\def \x {\mathbf{x}}
\def \y {\mathbf{y}}

\author{Sebasti\'an Donoso}
\author{Anh N. Le}
\author{Joel Moreira}
\author{Wenbo Sun}
\address{Instituto de Ciencias de la Ingenier\'ia\\ Universidad de O'Higgings\\ Av. Libertador Bernardo O'Higgins 611, Rancagua, Chile }
\email{sebastian.donoso@uoh.cl}
\address{Department of Mathematics\\
	Northwestern University\\
	2033 Sheridan Road, Evanston, IL 60208-2730, USA}
\email{anhle@math.northwestern.edu  }
\address{ Department of Mathematics\\
	Northwestern University\\
	2033 Sheridan Road, Evanston, IL 60208-2730, USA }
\email{joel.moreira@northwestern.edu}
\address{Department of Mathematics\\ The Ohio State University\\ 231 West 18th Avenue,
	Columbus OH, 43210-1174, USA}
\email{sun.1991@osu.edu}

\title{Optimal lower bounds for multiple recurrence}

\begin{document}

\begin{abstract}
Let $(X,{\mathcal B},\mu,T)$ be an ergodic measure preserving system, $A \in \mathcal{B}$ and $\epsilon>0$. We study the largeness of sets of the form
\begin{equation*}
\begin{split}
S = \big\{n\in\mathbb{N}\colon\mu(A\cap T^{-f_1(n)}A\cap T^{-f_2(n)}A\cap\ldots\cap T^{-f_k(n)}A)> \mu(A)^{k+1} - \epsilon \big\}
\end{split}
\end{equation*}
for various families $\{f_1,\dots,f_k\}$ of sequences $f_i\colon \N \to \N$.

For $k\leq 3$ and $f_{i}(n)=if(n)$, we show that $S$ has positive density if $f(n)=q(p_n)$ where $q\in\Z[x]$ satisfies $q(1)$ or $q(-1) =0$ and $p_n$ denotes the $n$-th prime; or when $f$ is a certain Hardy field sequence.
If $T^q$ is ergodic for some $q \in \mathbb{N}$, then for all $r \in \mathbb{Z}$, $S$ is syndetic if $f(n) = qn + r$.

For $f_{i}(n)=a_{i}n$, where $a_{i}$ are distinct integers, we show that $S$ can be empty for $k\geq 4$, and for $k = 3$ we found an interesting relation between the largeness of  $S$ and the abundance of solutions to certain linear equations in sparse sets of integers. We also provide some partial results when the $f_{i}$ are distinct polynomials.
\end{abstract}

\maketitle

\section{Introduction}
\subsection{Historical background}
The classical Poincar\'e recurrence theorem states that for every measure preserving system $(X,{\mathcal B},\mu,T)$ and every set $A\in{\mathcal B}$ with $\mu(A)>0$, there exists some $n\in\N$ such that $\mu(A\cap T^{-n}A)>0$. This result was improved by Khintchine in \cite{Khintchine34}, who showed that under the same conditions, for every $\epsilon>0$, the set
$$S:=\big\{n \in \N \colon \mu(A\cap T^{-n}A)>\mu(A)^2-\epsilon\big\}$$
is \emph{syndetic}, meaning that it has bounded gaps.
Taking a mixing system, one sees that the bound $\mu(A)^2$ is optimal.

In \cite{Furstenberg77}, Furstenberg established a multiple recurrence theorem, showing that for every measure preserving system $(X,{\mathcal B},\mu,T)$,  every $k\in\N$ and  set $A\in\mathcal{B}$ with $\mu(A)>0$, there exists a syndetic set $S\subset\N$ such that for all $n\in S$, we have
\begin{equation}\label{eq_intro1}
\mu(A\cap T^{-n}A\cap\cdots\cap T^{-kn}A)>0.
\end{equation}

One could hope to improve Furstenberg's multiple recurrence theorem in the same way that Khintchine's theorem strengthens Poincar\'e's.
Since for a system mixing of all orders, the left hand side of \eqref{eq_intro1} approaches $\mu(A)^{k+1}$ as $n\to\infty$, one could hope that under the same conditions as Furstenberg's multiple recurrence theorem, for every $\epsilon>0$, the set
\begin{equation}\label{eq_intro2}
\big\{n\in \N \colon\mu(A\cap T^{-n}A\cap T^{-2n}A\cap\cdots\cap T^{-kn}A)>\mu(A)^{k+1}-\epsilon\big\}
\end{equation}
is syndetic.
This was showed to be true by Furstenberg when the system is weakly mixing. The general case was finally settled by Bergelson, Host and Kra in \cite{Bergelson_Host_Kra05},
who showed that if the system $(X,{\mathcal B},\mu,T)$ is ergodic, then the set in \eqref{eq_intro2} is syndetic when $k=1,2,3$ (with the case $k=1$ following from Khintchine's theorem).
However, the set in \eqref{eq_intro2} may be empty if the system is not ergodic or if $k\geq4$:
\begin{theorem}[{\cite[Theorems 2.1 and 1.3]{Bergelson_Host_Kra05}}]\label{thm_bhknegative}
	There exist a non-ergodic measure preserving system $(X,{\mathcal B},\mu,T)$ and for each $\ell\in\N$ a set $A\in{\mathcal B}$ with $\mu(A)>0$ such that for every $n \in \N\setminus \{0 \}$
	$$\mu(A\cap T^{-n}A\cap T^{-2n}A)< \mu(A)^\ell$$
	There exist a totally ergodic measure preserving system $(X,\mathcal{B},\mu,T)$ and for each $\ell\in\N$ a set $A\in{\mathcal B}$ with $\mu(A)>0$ such that for every $n \in \N\setminus \{0 \}$
	$$\mu(A\cap T^{-n}A\cap \cdots\cap T^{-4n}A)< \mu(A)^\ell$$
\end{theorem}
The first part of this theorem explains why one needs to focus on ergodic systems when studying optimal multiple recurrence.

Furstenberg's multiple recurrence theorem has been extended in several different directions, each leading to the question of whether (or under which conditions) optimal recurrence can be achieved.
In this paper, we are mostly concerned with expressions of the form
\[
    \mu(A \cap T^{-f_1(n)} A \cap T^{-f_2(n)} A \cap \ldots \cap T^{-f_k(n)}A)
\]
for various families $\{f_1,\dots,f_k\}$ of functions $f_i\colon \N\to\Z$.
In most cases where recurrence has been established, optimal recurrence can be obtained for weakly mixing systems (see \cite{Bergelson87} when the $f_i$ are polynomials and \cite{Bergelson_Knutson09,Frantzikinakis10,Frantzikinakis15} for more general $f_i$), or when the functions are ``independent'' (see \cite{Frantzikinakis_Kra05b,Frantzikinakis_Kra06} for the case of linearly independent polynomials and \cite{Frantzikinakis15} for more general $f_i$ with different growth).
In the general case, besides the aforementioned paper \cite{Bergelson_Host_Kra05}, the main progress was obtained by Frantzikinakis in \cite{Frantzikinakis08}, where the case $k \leq 3$ and the $f_i$ are polynomials is studied in detail.

\subsection{Optimal recurrence along \texorpdfstring{$\mathbf{\{f(n),2f(n),\dots,kf(n)\}}$}{(f(n), 2f(n),..., kf(n))}}

Our first result concerns the sequence of prime numbers and answers a question of Kra.
Multiple recurrence along polynomials evaluated at primes was established by Frantzikinakis, Host and Kra in \cite{Frantzikinakis_Host_Kra07,Frantzikinakis_Host_Kra13} and Wooley and Ziegler in \cite{Wooley_Ziegler12}.
Our result states that one can also obtain optimal recurrence in this setting.
\begin{theorem}\label{thm_primesOR}
	Let $(p_n)_{n\in\N}$ be the increasing enumeration of the primes. Let $(X,{\mathcal B},\mu,T)$ be an ergodic measure preserving system and let $f\in\Z[x]$ be such that $f(1)=0$ or $f(-1) = 0$.
	Then for every $A\in{\mathcal B}$, $\epsilon>0$ and $k\in\{1,2,3\}$, the set
	\begin{equation}\label{eq_thm_primesOR}
	\big\{n\in\mathbb{N}\colon\mu(A\cap T^{- f(p_n)}A\cap T^{- 2f(p_n)}A\cap\cdots\cap T^{-kf(p_n)}A)>\mu(A)^{k+1}-\epsilon\big\}
	\end{equation}
	has positive lower density\footnote{The lower density of a set $E\subset\N$ is defined to be $\underline{d}(E)=\liminf\limits_{N\to\infty}\frac{|E\cap\{1,\dots,N\}|}N$.}.
\end{theorem}


We remark that the set in \eqref{eq_thm_primesOR} is not syndetic in general. In fact, it follows from \cite{Shiu00} that when $f(x)=x-1$, for every non-trivial finite system, there exists $A\in{\mathcal B}$ such that the set in \eqref{eq_thm_primesOR} has unbounded gaps.

A similar result can be obtained if $f(p_n)$ is replaced with $\lfloor f(n)\rfloor$, where $f$ is a function belonging to a Hardy field with polynomial growth and sufficiently far away from $\Q[x]$ and $\lfloor \cdot \rfloor$ indicates taking integer part. More precisely,
denote by $\mathcal{G}$ the set of all equivalence classes of smooth functions $\R\to\R$, where $f \sim g$ if there exists a constant $c > 0$ such that $f(x) = g(x)$ for all $x > c$. A \emph{Hardy field} is a subfield of the ring $(\mathcal{G}, +, \times)$ which is closed under differentiation.
Let $\mathcal{H}$ be the union of all Hardy fields.
We say that a function $f$ has polynomial growth if there exists $d \in \N$ such that $f(x)/x^d \to 0$ as $x \to \infty$.
Multiple recurrence for functions in $\mathcal{H}$ of polynomial growth was obtained by Frantzikinakis and Wierdl \cite{Frantzikinakis_Wierdl09,Frantzikinakis10}.

\begin{theorem}
	\label{theorem:glob-Hardy}
	Let $f \in \mathcal{H}$ have polynomial growth and satisfy $\big|f(x)-cq(x)\big|/\log(x)\to\infty$ for every $c\in\R$, $q\in\Z[x]$.
	Then for every ergodic measure preserving system $(X,{\mathcal B},\mu,T)$, $A\in{\mathcal B}$, $\epsilon>0$ and  $k\in\{1,2,3\}$, the set
	\begin{equation}
	\label{eq_thm_HardyOR}
	    \big\{n\in\mathbb{N}\colon\mu(A\cap T^{-\lfloor f(n)\rfloor} A \cap T^{-2\lfloor f(n)\rfloor} A \cap\cdots\cap T^{-k\lfloor f(n)\rfloor}A)>\mu(A)^{k+1}-\epsilon\big\}
	\end{equation}
	has positive lower density.
\end{theorem}

Examples of functions that satisfy the conditions in \cref{theorem:glob-Hardy} are $f(x)=x^{c}$ where $c > 0$, $c \not \in \mathbb{Z}$, $f(x)=x \log x$, $f(x)=x^2 \sqrt{2} + x \sqrt{3}$, and $f(x)=x^3 + (\log x)^3$.

We point out that in \cref{theorem:glob-Hardy}, we cannot replace ``has positive density'' by ``is syndetic''. This is easy to see for certain functions $f$ growing slowly (for instance $f(x)=x^c$ when $c<1$).
For such functions, $\lfloor f(n) \rfloor$ is constant in arbitrarily long intervals and takes every value which is large enough.
Therefore there are gaps of the set \eqref{eq_thm_HardyOR} which are arbitrarily long.

On the other hand, we expect the set in \eqref{eq_thm_HardyOR} to be \emph{thick}, i.e. to contain arbitrarily long intervals.
Some evidence in this direction is given in \cite{Bergelson_Moreira_Richter18}, where the set $\big\{n\in\mathbb{N}\colon\mu(A\cap T^{-\lfloor f(n)\rfloor}A\cap\cdots\cap T^{-k\lfloor f(n)\rfloor}A)>0\big\}$ is shown to be thick, as well as the set in \eqref{eq_thm_HardyOR} when $k=1$.

Our third result concerns sequences of the form $f(n)=qn+r$ for fixed $q,r\in\Z$ and was suggested by Kra.
\begin{theorem}\label{thm_progressionOR}
	Let $q,r\in\Z$, with $q>0$, and $(X,{\mathcal B},\mu,T)$ be a measure preserving system with $T^q$ ergodic. Let $A\in{\mathcal B}$, $\epsilon>0$ and $k\in\{1,2,3\}$.
	Then the set
	\begin{equation}\label{eq_thm_progressionOR}
\big\{n\in\mathbb{N}\colon\mu(A\cap T^{-(qn+r)}A\cap T^{-2(qn+r)}A\cap\cdots\cap T^{-k(qn+r)}A)>\mu(A)^{k+1}-\epsilon\big\}
	\end{equation}
	is syndetic.
\end{theorem}

Observe that the conclusion of this theorem is equivalent to the statement that the intersection
$$\big\{n\in\mathbb{N}\colon\mu(A\cap T^{-n}A\cap T^{-2n}A\cap\cdots\cap T^{-kn}A)>\mu(A)^{k+1}-\epsilon\big\}\cap\big(q\Z+r\big)$$
is syndetic.
By looking at the rotation on $q$ elements, we see that the hypothesis $T^q$ is ergodic is necessary.

If in \eqref{eq_thm_progressionOR} one replaces the optimal lower bound $\mu(A)^{k+1}-\epsilon$ with $0$, then the set is syndetic for any $k\in\N$.
This was proved in \cite{Host_Kra02} for $k=2$ and $k=3$, and for larger $k$, this is essentially the content of \cite[Corollary 6.5]{Frantzikinakis04}; see also \cite{Moreira_Richter19}.

Theorems \ref{thm_primesOR}, \ref{theorem:glob-Hardy} and \ref{thm_progressionOR} will be proved in \cref{s3}.

\subsection{Optimal recurrence along \texorpdfstring{$\mathbf{\{a_1 n, a_2 n, \ldots, a_k n\}}$}{(a1 n, a2 n,..., ak n)}}

Next we study optimal recurrence for the expression
$$\mu(T^{-a_1 n}A\cap T^{- a_2 n}A\cap\cdots\cap T^{- a_k n}A)$$
where $a_1,\dots,a_k$ are distinct integers.
In particular, if $a_i=i$, then the results of Bergelson, Host and Kra tell us that we have optimal recurrence if and only if $k \leq 4$.
More generally, in \cite{Frantzikinakis08} it is proved that if $k \leq 3$, or $k=4$ and $a_2+a_3=a_1+a_4$, then optimal recurrence holds, but any other case is not known.

Expanding  an argument of Ruzsa, presented in the appendix of \cite{Bergelson_Host_Kra05}, we prove that for $k \geq 5$, one does not have optimal recurrence along $\{a_1 n, a_2 n, \ldots, a_k n\}$.
\begin{theorem}
	\label{thm_BadRecurrence}
	Let $a_1< \ldots < a_5 \in \N$. Then there exists an ergodic system $(X,\mathcal{B}, \mu, T)$ such that for every $\ell>0$, there exists a set $A\in\mathcal{B}$ with $\mu(A)>0$ and
	\[
	\mu(T^{-a_1 n} A \cap T^{-a_2 n} A \cap \ldots \cap T^{-a_5 n} A)< \mu(A)^\ell
	\]
	for every $n \in \N\setminus \{0\}$.
\end{theorem}

The cases not covered by the above results seem difficult to address.
For instance, it is not known whether for every ergodic system, every set $A$ and $\epsilon>0$, the set
\[
    \big\{n \in \N\setminus \{0\}\colon  \mu(A\cap T^{-2n} A\cap T^{-3n}A\cap T^{-4n}A)>\mu(A)^4-\epsilon\big\}
\]
is non-empty or not (let alone syndetic).
In \cite{Frantzikinakis08}, Frantzikinakis showed that a positive answer to this question would imply the existence of solutions to certain linear equations in sparse sets.
We tighten this relationship between optimal recurrence along $\{a_1 n, a_2 n, a_3 n, a_4 n\}$ and abundance of solutions to certain linear equations in sparse sets, by showing that these two phenomenons are essentially equivalent.
To formulate our result, we need to introduce some notation.

\begin{definition}
\label{definition_densities}
	Let $m,d,N\in\N$. Denote $[N]:=\{0,1,\dots,N-1\}$.
	Given a set $E\subseteq [N]^{m}$ and a subspace $V\subseteq\mathbb{Q}^{d\times m}$, define
	\[
	    d_{m,N}(E)=\frac{|E|}{N^m} \qquad \mbox{and} \qquad
	    D_{m,N}(V,E)=\frac{|V\cap E^d|}{|V\cap[N]^{d\times m}|}.
	\]
\end{definition}
 Observe that a point $(x_1,\dots,x_{dm})\in[N]^{d\times m}$ belongs to $V$ if and only if the coordinates $x_1,\dots,x_{dm}$ satisfy some system of linear equations.
The reader should think of $D_{m,N}(V,E)$ as the proportion of solutions to that system with all variables in $E$.

\begin{definition}
  A subset $S\subseteq\mathbb{N}$ is a \emph{Bohr$_{0}$ set} if there exist a compact abelian group $Z$, $\delta>0$ and $\alpha\in Z$ such that $S=\{n\in\N\colon n\alpha\in B(\delta)\}$, where $B(\delta)$ is the ball in $Z$ centered at the identity $0_Z$ with radius $\delta$.
  If $Z$ is a finite dimensional torus, we say that $S$ has \emph{finite rank}.
\end{definition}

\begin{theorem}\label{thm:integerdirect}
    Let $a_{1},\dots,a_{4} \in \N$ be distinct and $V$ be the subspace of $\mathbb{Q}^4$ spanned by $(a^{i}_{1},\dots,a^{i}_{4})$ for $0\leq i\leq 2$.\footnote{We adopt the convention that if some $a_{i}$ equals to 0, then $a_{i}^{0}=0^{0}=1$.}
    Suppose there exist $C>0$ and $\ell \geq 4$ such that for every $m\in\mathbb{N}$, every sufficiently large $N$ and subset $E\subseteq[N]^{m}$, we have $D_{m,N}(V^m,E)> Cd_{m,N}(E)^\ell$.
    Then for every ergodic system $(X,\mathcal{B},\mu,T)$
	and $A\in\mathcal{B}$ with $\mu(A)>0$, there exists a Bohr$_{0}$ set $S\subseteq\mathbb{N}$ with finite rank such that
	\[
	    \limsup_{N-M\to\infty}\frac{1}{\vert S\cap [M,N)\vert}\sum_{n\in S\cap [M,N)}\mu(T^{-a_{1}n}A \cap \dots \cap T^{-a_{4}n}A)\geq C\left(1-\frac4\ell\right)^\ell\mu(A)^{\ell}.
	\]

\end{theorem}
\begin{remark*}
    It is easy to see that the conclusion of \cref{thm:integerdirect} implies that the set
    \[
    \label{1234}
    \left\{n\in\mathbb{N}\colon \mu(T^{-a_{1}n}A\cap\dots\cap T^{-a_{4}n}A) > C\left(1-\frac4\ell\right)^\ell\mu(A)^{\ell}-\epsilon\right\}
    \]
    is syndetic for all $\epsilon>0$.
\end{remark*}

We have a partial converse to \cref{thm:integerdirect}.
\begin{theorem}\label{thm:integerinverse}
	Let $a_{1},\dots,a_{4}\in \N$ be distinct and
	$V$ be the subspace of $\mathbb{Q}^4$ spanned by $(a^{i}_{1},\dots,a^{i}_{4})$ for $0\leq i\leq 2$.
Suppose there exist $C > 0$ and $\ell \geq 4$ such that for every ergodic system $(X,\mathcal{B},\mu,T)$
		 and $A\in\mathcal{B}$ with $\mu(A)>0$, there exists a Bohr$_{0}$ set $S\subseteq\mathbb{N}$ with finite rank such that
		$$\limsup_{N-M\to\infty}\frac{1}{\vert S\cap [M,N)\vert}\sum_{n\in S\cap [M,N)}\mu(T^{-a_{1}n}A\cap\dots\cap T^{-a_{4}n}A)\geq C\mu(A)^{\ell}.$$
Then for every $m\in\mathbb{N}$, $E\subseteq[N]^{m}$, and sufficiently large $N$, \[
    D_{m,N}(V^m,E)\geq C\beta^md_{m,N}(E)^\ell,
\]
where $\beta>0$ is an explicit constant depending only on $a_1,\dots,a_4$ and $\ell$.
\end{theorem}


We provide some examples to illustrate Theorems \ref{thm:integerdirect} and \ref{thm:integerinverse}.
\begin{example}
	Suppose $(a_{1},a_{2},a_{3},a_{4})=(0,2,3,4)$. In this case $V$ is the $\mathbb{Q}$-span of $(1,1,1,1)$, $(0,2,3,4)$ and $(0,4,9,16)$. It follows that
	$$V=\big\{(x,y,z,w)\in\mathbb{Q}^{4}\colon x-6y+8z-3w=0\big\}.$$
	For convenience set $m=1$.
Then $D_{1,N}(V,E)$ is essentially the density of solutions of the equation $x-6y+8z-3w=0$ in $E$, i.e. the proportion of tuples $(x,y,z,w)\in[N]^{4}$ satisfying $x-6y+8z-3w=0$ that belong to $E^4$.
The hypothesis in \cref{thm:integerdirect} can be rephrased informally as saying that this density can be bounded from below by the $\ell$-th power of the density $d_{1,N}(E)$ of the set $E$.
\end{example}

\begin{example}
	Suppose that $a_{1}+a_{2}=a_{3}+a_{4}$. In this case, an elementary computation shows that
	$$V=\big\{(x,y,z,w)\in\mathbb{Q}^{4}\colon s(x-y)+t(z-w)=0\big\},$$
where $s=a_3-a_4$ and $t=a_2-a_1$. We can assume, without loss of generality, that both $s$ and $t$ are positive.

Given $N,m\in\N$ and a set $E\subset[N]^m$, denote by $P(n)$ the number of pairs $(x,z)\in E^2$ satisfying $sx+tz=n$ for each $n\in\N^m$.
Observe that if $(x,y,z,w)\in E^4\cap V^m$ then $sx+tz=sy+tw\in\big[(s+t)N\big]^m$.
We have
$$\sum_{n\in\big[(s+t)N\big]^m}P(n)=|E|^2\qquad\text{ and }\qquad \sum_{n\in \big[(s+t)N\big]^m}P(n)^2=|E^4\cap V^m|.$$
It follows from the Cauchy-Schwarz inequality that $\frac{|E|^4}{N^m(s+t)^m}\leq|E^4\cap V^m|$, which in turn implies that
\begin{equation}\label{eq_easyFrantzikinakis+us}
D_{m,N}(V,E) \geq \beta^md_{m,N}(E)^4
\end{equation}
where $0<\beta<\lim\limits_{N\to\infty}\frac{N^3}{|V\cap[N]^4|(s+t)}$ (it is easy to see that the limit exists and is positive).

By \cref{thm:integerinverse}, the conclusion \eqref{eq_easyFrantzikinakis+us} also follows from \cite[Theorem C]{Frantzikinakis08}.

\end{example}

\subsection{Optimal recurrence along a family of polynomials}

In \cite{Frantzikinakis08}, Frantzikinakis studied in detail optimal recurrence for $(0, p_1(n), p_2(n), p_3(n))$ where $p_1, p_2, p_3 \in \Z[x]$ and dealt with most cases in that regime. However, some stubborn questions remain unanswered.
For instance, it is not known if there exists $\ell>0$ such that for every ergodic system $(X,\mathcal{B},\mu,T)$, $A\in\mathcal{B}$ and $\epsilon>0$, the set
\begin{equation}\label{eq_globpolynomial}
  \big\{n\in\N \setminus\{0\} \colon\mu(A\cap T^{-n}A \cap T^{-2n}A \cap T^{-n^2}A)>\mu(A)^{\ell}-\epsilon\big\}
\end{equation}
is non-empty (let alone syndetic).

On the other hand, if we impose some particular conditions on the system $(X, \mathcal{B}, \mu, T)$, the set \eqref{eq_globpolynomial} with $\ell = 4$ is indeed syndetic.

\begin{proposition}
	\label{proposition:globn2connected}
	Let $(G/\Gamma, \mathcal{B}, \mu, T)$ be an ergodic nilsystem where $G$ is connected (see \cref{section:background} for the definitions).	Then for every $A \in \mathcal{B}$ and $\epsilon > 0$, the set
	\begin{equation}\nonumber
	\big\{n\in\N\colon\mu(A\cap T^{-n}A \cap T^{-2n}A \cap T^{-n^2}A)>\mu(A)^{4}-\epsilon\big\}
	\end{equation}
	is syndetic.
\end{proposition}

\begin{remark*}
In \cref{s5} we formulate and prove a generalization of \cref{proposition:globn2connected} for more general ergodic systems $(X, \mathcal{B}, \mu, T)$ that considers the representation of its 3-step nilfactor (see \cref{section:background} for the definition).
\end{remark*}


We are unable to remove the connectedness assumption. Hence the general question regarding optimal recurrence for $(0, n, 2n, n^2)$ remains open.
However in the next result, we provide an example of lack of optimal recurrence for this family in the case of two commuting transformations (for other results on optimal recurrence for commuting transformations, see also \cite{Chu11,Donoso_Sun2018}).

\begin{proposition}\label{proposition:commutingpoly}
	There exists a  system $(X,\mathcal{B}, \mu,T_1, T_2)$, with $T_1$ ergodic and $T_1T_2=T_2T_1$  such that  for every integer $\ell>0$, there exists $A\in\mathcal{B}$ with $\mu(A)>0$ such that
	\[
	\mu(A\cap T_1^{- n} A \cap T_1^{-2n} A \cap T_2^{-n^2} A) < \mu(A)^\ell
	\]
	for every $n \in \N\setminus \{0\}$.
\end{proposition}

Theorems \ref{thm_BadRecurrence}, \ref{thm:integerdirect} and \ref{thm:integerinverse} are proved in \cref{s4}, while Propositions
\ref{proposition:globn2connected} and \ref{proposition:commutingpoly} will be proved in \cref{s5}.

\subsection{Acknowledgment}
We are grateful to Bryna Kra for proposing the questions studied in this paper and providing many important feedback. We thank the Mathematics Research Communities program of the AMS for giving us the optimal environment to start this project and the anonymous referee for helpful suggestions. S. Donoso is supported by Fondecyt Iniciaci\'on Grant 11160061 and J. Moreira is supported by the NSF via grant DMS-1700147.

\section{Background}
\label{section:background}
\subsection{Nilmanifolds, nilsystems and nilsequences}
\label{subsec:nilseq}
Given a group $G$, we denote its lower central series by $G=G_1 \triangleright G_2\triangleright\cdots$, where each term is defined by $G_{i+1}=[G_i,G]$, i.e., $G_{i+1}$ is the subgroup of $G$ generated by all the commutators $[a,b]:=aba^{-1}b^{-1}$ with $a\in G_i$ and $b\in G$.
The group $G$ is a \emph{$k$-step nilpotent group} if $G_{k+1}$ is the trivial group.

Let $G$ be a $k$-step nilpotent Lie group and let $\Gamma$ be a \emph{uniform} (i.e closed and cocompact) subgroup of $G$.
The compact homogeneous space $X := G/\Gamma$ is called a \emph{$k$-step nilmanifold}.
Let $\pi \colon G \rightarrow X$ be the standard quotient map.
We write $1_X = \pi(1_G)$ where $1_G$ is the identity element of $G$.
Denote by $G^0$ the connected component of $G$ containing the identity $1_G$.
If $X$ is connected, then $X = \pi(G^0)$.

The space $X$ is endowed with a unique probability measure that is invariant under translations by $G$. This measure is called the \emph{Haar measure} for $X$, and denoted by $\mu_X$. For every $\tau \in G$, the measure preserving system $(X,\mathcal{B}, \mu_X, T)$ given by $Tx=\tau\cdot x, x\in X$ is called a \emph{$k$-step nilsystem}, where $\mathcal{B}$ is the Borel $\sigma$-algebra of $X$.

Let $C(X)$ denote the set of continuous functions on $X$. For $F \in C(X)$ and $x \in X$, the sequence $\psi(n) := F(T^n x)$ is called a \emph{basic $k$-step nilsequence}. A \emph{$k$-step nilsequence} is a uniform limit of basic $k$-step nilsequences.

We say that a sequence $(x_n)_{n \in \mathbb{N}}$ is \emph{equidistributed} on a nilmanifold $X$ if for every $F \in C(X)$, we have
\[
\lim_{N \to \infty} \frac{1}{N} \sum_{n=1}^N F(x_n) = \int_X F \, d\mu_X.
\]
On the other hand, we say that $(x_n)_{n\in\N}$ is \emph{well distributed} on $X$ if
\[
\lim_{N - M \to \infty} \frac{1}{N - M} \sum_{n = M}^{N-1} F(x_n) = \int_X F \, d \mu_X.
\]
for all $F \in C(X)$.

\subsection{Nilfactors}
\label{sec_nilfactors}

Let $(X,\mathcal{B}, \mu, T)$ be an ergodic measure preserving system, let $k\in\N$ and, for each $j=1,\dots,k$, let $(s_j(n))_{n \in \mathbb{N}}$ be an integer valued sequence and $f_j \in L^{\infty}(\mu)$. Then a factor $(Y, \mathcal{D}, \nu, S)$ of $X$ is said to be \emph{characteristic for the expression $\frac1N\sum_{n=1}^N \prod_{j=1}^k T^{s_j(n)} f_j$} if
\[
\lim_{N \rightarrow \infty} \left( \frac1N\sum_{n=1}^N \prod_{j=1}^k T^{s_j(n)} f_j -  \frac1N\sum_{n=1}^N \prod_{j=1}^k T^{s_j(n)} \mathbb{E}\left(f_j|Y \right) \right) = 0,
\]
where $\mathbb{E}\left(f|Y \right)$ denotes the conditional expectation of $f$ onto $Y$ and the limit is taken in $L^2(\mu)$.
Host and Kra \cite{Host_Kra05} showed that there exists a characteristic factor for $ \frac{1}{N}\sum_{n=1}^N \prod_{j=1}^k T^{jn} f_j$ which is an inverse limit of $(k-1)$-step nilsystems (see also \cite{Ziegler07}). We call this factor the \emph{$(k-1)$-step nilfactor of $X$} and denote it by $Z_{k-1}(X)$ (or $Z_{k-1}$ when the system is clear).

If $(X = G/\Gamma, \mathcal{B}, \mu, T)$ is an ergodic $k$-step nilsystem with $G$ being a $k$-step nilpotent Lie group with the lower central series
\[
    G = G_1 \triangleright G_2 \triangleright \ldots \triangleright G_k \triangleright G_{k+1} = \{1\},
\]
then the $s$-step nilfactor of $X$ is $G/G_{s+1} \Gamma$ and can also be represented as $(G/G_{s+1})/(G_{s+1} \Gamma/G_{s+1}))$ (see \cite[Chapter 11]{Host_Kra18}). Note that if $G$ is connected, then $G/G_{s+1}$ is also connected.

\subsection{Limit formula for multiple averages on nilsystems}
The following description of the limiting distribution of multiple ergodic averages in nilsystems is essentially due to Ziegler \cite{Ziegler05}.

\begin{theorem}
	\label{lemma:ziegler}
	Let $a_1,\dots,a_d\in\Z$ be distinct. Let $(X = G/\Gamma,\mathcal{B}, \mu, T)$ be a $k$-step ergodic nilsystem and let $f_1, f_2, \ldots, f_d \in L^{\infty}(\mu)$.
	For each $i \in \{1,\dots,k\}$, let $\Gamma_i = \Gamma \cap G_i$ and let $\mu_i$ be the Haar measure on $G_i/\Gamma_i$.
	Then for $\mu$-a.e. $x = g \Gamma \in X$,
	\begin{multline}
	\label{eq_ziegleroriginal}
	\displaystyle \lim_{N - M \to \infty} \frac{1}{N-M} \sum_{n=M}^{N-1} f_1 (T^{a_1 n} x) \dots f_d (T^{a_d n} x)  =\\
	\displaystyle \int_{G_1/\Gamma_1} \int_{G_2/\Gamma_2} \ldots \int_{G_k/\Gamma_k} \prod_{i=1}^df_i \left(g g_1^{\binom{a_i}1} \dots g_k^{\binom{a_i}k} \Gamma\right) \, d \mu_k (g_k \Gamma_k) \ldots d \mu_2(g_2 \Gamma_2) d \mu_1 (g_1 \Gamma_1).
	\end{multline}
\end{theorem}

\begin{remark*}
	\cref{lemma:ziegler} in particular asserts that the right hand side of \eqref{eq_ziegleroriginal} does not depend on the choice of representative $g_i$ for the co-set $g_i\Gamma_i$.

	The statement in \cite[Theorem 1.2]{Ziegler05} requires $G$ to be connected and simply connected.
	These restrictions were removed in \cite[Theorem 5.4]{Bergelson_Host_Kra05}, although in that paper the limit is described in a different (but equivalent) form; see also \cite[Theorem 6.3]{Leibman10b}.
\end{remark*}

Let $(X=G/\Gamma, \mathcal{B}, \mu, T)$ be an ergodic nilsystem with $Tx = \tau \cdot x$ for some $\tau \in G$.
Then its $1$-step nilfactor (also called the \emph{Kronecker factor}) is $(Z_1,\alpha)$ where $Z_1 = G/G_2 \Gamma$ and $\alpha = \tau G_2\Gamma\in Z_1$.
Observe that $Z_1$ is a finite dimensional torus; we will use additive notation for the group operation in $Z_1$.
Define the Bohr$_{0}$ set
\begin{equation}
\label{equ:s}
\begin{split}
    S_{\delta}:=\{n\in\N\colon n \alpha \in B(\delta)\},
\end{split}
\end{equation}
where $B(\delta)$ is the ball in $Z_1$ centered at 0 with radius $\delta$.
By ergodicity (and hence unique ergodicity), the \emph{uniform density} $d(S_{\delta})$ of $S_{\delta}$ is
$$
d(S_{\delta}) := \lim_{N-M\to\infty}\frac{\vert S_{\delta}\cap[M,N)\vert}{N-M}=\mu_{Z_1}(B(\delta)),
$$
where $\mu_{Z_1}$ is the Haar measure on $Z_1$.

We need the following proposition, whose proof for case $d=3$ is sketched in \cite[Page 35]{Frantzikinakis08}.
The proof for general $d$ is similar and included here for completeness.
\begin{proposition}
	\label{prop:sz}
	Assume the hypothesis as in \cref{lemma:ziegler}. For each $\delta > 0$, let $S_{\delta}$ be defined by \eqref{equ:s}.
	Then for $\mu$-almost every $x = g \Gamma \in X$,
	\begin{multline*}
	\lim_{\delta \to 0}\lim_{N-M \to \infty}\frac{1}{| S_{\delta} \cap [M, N)|}
	\sum_{n \in S_{\delta} \cap [M,N)} f_1 (T^{a_1 n} x) \dots f_d (T^{a_d n} x) = \\
	\int_{G_2 /\Gamma_2} \int_{G_2 /\Gamma_2} \int_{G_3 /\Gamma_3} \ldots \int_{G_k/\Gamma_k} \prod_{i=1}^df_i\left(g g_1^{a_i \choose 1} \dots g_k^{a_i \choose k} \Gamma\right)\, d \mu_k(g_k \Gamma_k) \ldots d \mu_2(g_2 \Gamma_1)d \mu_2(g_1 \Gamma_1).
	\end{multline*}
\end{proposition}

\begin{proof}
	Let $\pi\colon X\to Z_1$ be the natural projection.
	For any character $\chi$ of the compact abelian group $Z_1= X/G_2$, the composition $\chi\circ\pi$ is in $L^\infty(\mu)$, and  $\chi\circ\pi(T^nx)=\chi\big(n\alpha+\pi(x)\big)$ for all $n\in\N$ and $x\in X$.
	On the other hand, $\chi\circ\pi(gh\Gamma)=\chi\circ\pi(g\Gamma)$ whenever $h\in G_2$.
	By \cref{lemma:ziegler}, for $\mu$-almost every $x = g \Gamma \in X$, we have
	\begin{multline}
	\label{equ:ZL3}
	\lim_{N-M \to \infty} \frac{1}{N-M} \sum_{n=M}^{N-1}  \chi\big(n \alpha + \pi(x)\big) \prod_{i=1}^d f_i(T^{a_i n} x)
	=\\
	\int_{G_1/\Gamma_1} \ldots \int_{G_k/\Gamma_k} \chi\big(\pi(g g_1 \Gamma)\big)  \prod_{i=1}^d f_i\left(g g_1^{a_i \choose 1} \dots g_k^{a_i \choose k} \Gamma\right) \, d \mu_k(g_k \Gamma_k) \ldots d \mu_1(g_1 \Gamma_1).
	\end{multline}
	
	As $\chi$ is a character of $Z$, we have $\chi\big(n \alpha + \pi(x)\big) = \chi(n \alpha) \chi\big(\pi(x)\big)$, and $\chi\big(\pi(g g_1 \Gamma)\big) = \chi\big(\pi(g \Gamma) \big)\chi\big(\pi(g_1 \Gamma)\big)$.
Recall that $x = g\Gamma$.
After canceling $\chi\big(\pi(x)\big)$ from both sides of \eqref{equ:ZL3}, we get:
	\begin{multline}
	\label{equ:ZL4}
	\lim_{N-M \to \infty} \frac{1}{N-M} \sum_{n=M}^{N-1}  \chi(n \alpha)  \prod_{i=1}^d f_i(T^{a_i n} x)
	= \\
	\int_{G_1/\Gamma_1} \ldots \int_{G_k/\Gamma_k}\chi\big(\pi(g_1\Gamma) \big)\prod_{i=1}^d f_i\left(g g_1^{a_i \choose 1} \dots g_k^{a_i \choose k} \Gamma\right) \, d \mu_k(g_k \Gamma_k) \ldots d \mu_1(g_1 \Gamma_1).
	\end{multline}
	
	We can approximate the Riemann integrable function $\mathbbm{1}_{B(\delta)}$ by finite linear combinations of characters, and so we can replace $\chi$ in \eqref{equ:ZL4} with $\one_{B(\delta)}$ to get:
	\begin{multline}
	\label{equ:ZL5}
	\lim_{N-M \to \infty} \frac{1}{N-M} \sum_{n=M}^{N-1}  \mathbbm{1}_{B(\delta)} (n \alpha)  \prod_{i=1}^d f_i(T^{a_i n} x)
	= \\
	\int_{G_1/\Gamma_1} \ldots \int_{G_k/\Gamma_k} \mathbbm{1}_{B(\delta)}\big(\pi(g_1\Gamma) \big)\prod_{i=1}^d f_i\left(g g_1^{a_i \choose 1} \dots g_k^{a_i \choose k} \Gamma\right) \, d \mu_k(g_k \Gamma_k) \ldots d \mu_1(g_1 \Gamma_1).
	\end{multline}
	
	The left hand side of \eqref{equ:ZL5} is equal to:
	\begin{equation*}
	m_Z\big(B(\delta)\big)\lim_{N-M \to \infty} \frac{1}{| S_{\delta} \cap [M, N)|} \sum_{n \in S_{\delta} \cap [M,N)} \prod_{i=1}^d f_i(T^{a_i n} x).
	\end{equation*}
	
	On the other hand, the right hand side of \eqref{equ:ZL5} is equal to:
	\begin{equation*}
	\int_{\pi^{-1}(B(\delta))} \int_{G_2/\Gamma_2}\ldots \int_{G_k/\Gamma_k} \prod_{i=1}^d f_i\left(g g_1^{a_i \choose 1} \dots g_k^{a_i \choose k} \Gamma\right)  \, d \mu_k(g_k \Gamma_k) \ldots d \mu_1(g_1 \Gamma_1).
	\end{equation*}
	Let $\mu_\delta$ be the probability measure on $X$ defined by
	$$\int_X f\,d\mu_\delta=\frac1{\mu_{Z_1}(B(\delta))}\int_{\pi^{-1}(B(\delta))}f\,d\mu_X\qquad\forall f\in C(X).$$
	Since $\mu_X$ is invariant under the action of $G$ (and hence of $G_2$) and the set $\pi^{-1}(B(\delta))$ is invariant under $G_2$, we have that $\mu_\delta$ is invariant under the action of $G_2$.
	Moreover, any limit point of $\{\mu_\delta\colon \delta>0\}$ is supported on $G_2/\Gamma_2$.
	This shows that $\mu_\delta\to\mu_{G_2/\Gamma_2}$ as $\delta\to0$, where $\mu_{G_2/\Gamma_2}$ is the Haar measure on $G_2/\Gamma_2$.
	
	Therefore, dividing both sides of \eqref{equ:ZL5} by $\mu_{Z_1}(B(\delta))$ and taking the limit as $\delta\to0$, we obtain the desired conclusion.
\end{proof}

\subsection{Characteristic factors along \texorpdfstring{Bohr$_{\mathbf{0}}$}{Bohr0} sets}

We also need the following proposition whose proof is sketched in \cite[Page 34]{Frantzikinakis08}.

\begin{proposition}
	\label{proposition:characteristic-factor-along-bohr-set}
	Let $(X, \mathcal{B}, \mu, T)$ be an ergodic system, denote by $Z_1$ its Kronecker factor, let $\pi:X\to Z_1$ be the factor map and let $\alpha\in Z_1$ be such that $\pi(Tx)=\pi(x)+\alpha$ for every $x\in X$.
Let $\delta > 0$ and define $S_{\delta}$ as in \eqref{equ:s}.
Let $a_1, a_2, a_3 \in \mathbb{\N}$ be distinct and $f_1, f_2, f_3 \in L^{\infty}(\mu)$. Furthermore, suppose that $\mathbb{E}(f_i | Z_2) = 0$ for some $1 \leq i \leq 3$. Then
	\begin{equation}
	\label{equation:char-bohr-1}
	\lim_{N - M \to \infty} \frac{1}{|S_{\delta} \cap [M, N)|} \sum_{n \in S_{\delta} \cap [M,N)} f_1(T^{a_1 n} x) f_2(T^{a_2 n}x ) f_3(T^{a_3 n}x) = 0,
	\end{equation}
	where the limit is taken in $L^{2}(\mu)$.
\end{proposition}

\begin{proof}
	Without loss of generality, we assume $\mathbb{E}(f_1|Z_2) = 0$.
	Let $L$ be the limit on the left hand side of (\ref{equation:char-bohr-1}) and $d(S_{\delta})$ be the Banach density of $S_{\delta}$. Then
	\begin{multline}
	d(S_{\delta}) L = \lim_{N - M \to \infty} \frac{1}{N-M} \sum_{n=M}^{N-1} \mathbbm{1}_{S_{\delta}}(n) f_1(T^{a_1 n}x) f_2(T^{a_2 n} x) f_3(T^{a_3 n}x) = \\
	\lim_{N - M \to \infty} \frac{1}{N-M}\sum_{n=M}^{N-1} \mathbbm{1}_{B(\delta)}(n \alpha) f_1(T^{a_1 n}x) f_2(T^{a_2 n} x) f_3(T^{a_3 n}x).
	\end{multline}
	Approximating the Riemann integrable function $\mathbbm{1}_{B(\delta)}$ by linear combinations of characters, it suffices to show
	\begin{equation}
	\label{equation:char-bohr-1.5}
	\lim_{N - M \to \infty} \frac{1}{N-M}\sum_{n=M}^{N-1} \chi (n \alpha) f_1(T^{a_1 n}x) f_2(T^{a_2 n} x) f_3(T^{a_3 n}x) = 0
	\end{equation}
	for all character $\chi$ of $Z_1$. Note that the limit in the left hand side of (\ref{equation:char-bohr-1.5}) 
	is equal to
	\begin{equation}
	\bar{\chi}\big(\pi(x)\big) \lim_{N - M \to \infty} \frac{1}{N-M}\sum_{n=M}^{N-1} \chi\circ\pi\big(T^nx\big) f_1(T^{a_1 n}x) f_2(T^{a_2 n} x) f_3(T^{a_3 n}x).
	\end{equation}
	By \cite[Theorem 1.1 and 12.1]{Host_Kra05}, the above limit exists in $L^2(\mu)$ and does not change if we replace $f_i$ by $\mathbb{E}(f_i|Z_3)$.	
	Therefore, by approximation, we can assume that $(X, \mathcal{B},\mu, T)$ is a $3$-step nilsystem.
	
	 First suppose that $(X,\mathcal{B}, \mu, T)$ is totally ergodic.
Then the Kronecker factor $Z_1$ is connected and hence there exists $g \in Z_1$ such that $a_2 g = \alpha$.
Let $\alpha/a_2$ denote that element.
Consider the system $Y = (X \times Z_1, \mu \times\mu_{Z_1}, T \times \alpha/a_2)$.
Since $\mathbb{E}(f_1|Z_2(X)) = 0$, for almost every ergodic component $Y_t$ of $Y$, we have $\mathbb{E}(f_1 \otimes 1|Z_2(Y_t)) = 0$ (one way to verify this is to show that $\lVert f_1 \otimes 1 \rVert_{3} = 0$, where $\lVert \cdot \rVert_k$ is the Host-Kra's seminorm defined in \cite{Host_Kra05}). Hence by \cite[Theorem 12.1]{Host_Kra05},
	\begin{equation}
	\label{equation:char-bohr-23.1}
	    \lim_{N-M \to \infty} \frac{1}{N-M} \sum_{n=M}^{N-1} (T \times \alpha/a_2)^{a_1 n} f_1 \otimes 1 \cdot (T \times \alpha/a_2)^{a_2 n} f_2 \otimes \chi \cdot (T \times \alpha/a_2)^{a_3 n} f_3 \otimes 1 = 0,
	\end{equation}
	where the limit is taken in $L^2(\mu \times\mu_{Z_1})$. Rewriting the left hand side of (\ref{equation:char-bohr-23.1}), we get
	\begin{multline}
	\label{equation:char-bohr-23.2}
	    \lim_{N-M \to \infty} \frac{1}{N-M} \sum_{n=M}^{N-1} \chi(n \alpha + y) f_1(T^{a_1 n} x) f_2(T^{a_2 n} x) f_3(T^{a_3 n}x) =\\
	    \chi(y) \lim_{N-M \to \infty} \frac{1}{N-M} \sum_{n=M}^{N-1} \chi(n \alpha) f_1(T^{a_1 n} x) f_2(T^{a_2 n} x) f_3(T^{a_3 n} x) = 0.
	\end{multline}
Since $\chi(y) \neq 0$ for all $y \in G$, (\ref{equation:char-bohr-23.2}) implies (\ref{equation:char-bohr-1.5}).
	
	We now return to general situation without the total ergodicity assumption.
	Let $k$ be the number of connected components of $X$.
	Since $(X,\mathcal{B}, \mu, T)$ is ergodic, $(X, \mathcal{B},\mu, T^k)$ is totally ergodic.
	For all $0 \leq i \leq k-1$, applying the above argument with $T^k, T^{a_1 i} f_1, T^{a_2 i} f_2, T^{a_3 i} f_3$ replacing $T, f_1, f_2, f_3$, respectively, we get
	\[
	    \lim_{N - M \to \infty} \frac{1}{N-M}\sum_{n=M}^{N-1} \chi((kn + i) \alpha) f_1(T^{a_1(kn+i)} x) f_2(T^{a_2(kn+i)} x) f_3(T^{a_3(kn+i)} x) = 0
	\]
	for all character $\chi$ of $Z_1$.
	Taking the average over all $0 \leq i \leq k-1$, we derive (\ref{equation:char-bohr-1.5}). This finishes the proof.
\end{proof}

\section{Optimal recurrence along \texorpdfstring{$\{f(n), 2 f(n), \ldots, k f(n)\}$}{(f(n), 2f(n),..., kf(n))}}
\label{s3}

\subsection{Optimal recurrence along the sequence of shifted primes}
\label{section:positive-results}
We begin this section by recalling the following definition introduced in \cite{Frantzikinakis08}.
\begin{definition}
	A family of polynomials $\{g_1, g_2, g_3\}$ with $g_i \in \mathbb{Z}[x]$ for $i = 1, 2, 3$ is said to be of \emph{type $(e_1)$, $(e_2)$ or $(e_3)$} if some of its  permutation has the form
	\begin{enumerate}
		\item[$(e_1)$] $\{l q, mq, rq\}$ with $0 \leq l < m < r$ and $l + m \neq r$.
		\item[$(e_2)$] $\{lq, mq, kq^2 + rq\}$
		\item[$(e_3)$] $\{kq^2 + lq, kq^2 + mq, kq^2 + rq\}$
	\end{enumerate} for some $q \in \mathbb{Q}[x]$ and constants $k, l, m, r \in \mathbb{Z}$ with $k \neq 0$.
\end{definition}

\begin{theorem}
	\label{theorem:glob-primes}
	Let $(p_n)_{n \in \mathbb{N}}$ be the increasing enumeration of the primes. Let $(X, \mathcal{B}, \mu, T)$ be an ergodic measure preserving system, $A \in \mathcal{B}$ and $\epsilon > 0$. Suppose $g_1, g_2, g_3 \in \Z[x]$ with $g_i(0) = 0$ for $i = 1, 2, 3$. Then the sets
	\begin{equation*}
	\{n \in \mathbb{N}\colon \mu(A \cap T^{-g_1(p_n -1)} A) > \mu(A)^2 - \epsilon\}
	\end{equation*}
	and
	\begin{equation*}
	\{n \in \mathbb{N}\colon \mu(A \cap T^{-g_1(p_n -1)} A \cap T^{-g_2(p_n - 1)} A) > \mu(A)^3 - \epsilon\}
	\end{equation*}
	have positive lower density. Moreover, the set
	\begin{equation}
	\label{equation:glob-prime-3}
	\{n \in \mathbb{N}\colon \mu(A \cap T^{-g_1(p_n -1)} A \cap T^{-g_2(p_n - 1)} A \cap T^{-g_3(p_n -1)} A) > \mu(A)^4 - \epsilon\}
	\end{equation}
	also has positive lower density
	unless $g_1, g_2, g_3$ are pairwise distinct and $\{g_1, g_2, g_3\}$ is of type $(e_1), (e_2)$ or $(e_3)$.
	
	The same is true if $p_n - 1$ is replaced by $p_n + 1$.
\end{theorem}
\begin{proof}
	We only prove that the set in (\ref{equation:glob-prime-3}) has positive density as the proofs for the other two sets are similar.
	Fix $\epsilon > 0$ and  assume that the family $\{g_1, g_2, g_3\}$ is not of type $(e_1), (e_2)$ nor $(e_3)$. Denote
	\[
	\phi(n) = \mu(A \cap T^{-g_1(n)} A \cap T^{-g_2(n)} A \cap T^{-g_3(n)} A)
	\]
	for $n \in \mathbb{N}$.
	
	By \cite[Theorem 4.1]{Leibman10}, the sequence $(\phi(n))_{n \in \N}$ can be decomposed as $\phi(n) = \psi(n) + \delta(n)$, where $(\psi(n))_{n \in \N}$ is a nilsequence and
	\begin{equation}
	\label{equation:delta-is-null}
	\lim_{N-M \to \infty} \frac{1}{N-M} \sum_{n=M}^{N-1} |\delta(n)|=0.
	\end{equation}
	By \cite[Theorem 1.1]{Le17}, we also have
	\begin{equation}
	\label{equation:null-along-prime-minus-1}
	\lim_{N \to \infty} \frac{1}{N} \sum_{n=1}^N |\delta(p_n-1)| = 0.
	\end{equation}
	
	Since a nilsequence is a uniform limit of basic nilsequences, there exists a basic nilsequence $(F(\tau^n 1_Y))_{n \in \N}$ such that $|\psi(n) - F(\tau^n 1_Y)| < \epsilon/4$ for all $n \in \mathbb{N}$.
Here $F$ is a continuous function on a nilmanifold $Y = G/\Gamma$, $\tau \in G$ act ergodically on $Y$, and $1_Y = \Gamma \in Y$.
Assume that $Y$ has $d$ connected components and $Y_0$ is the component containing $1_Y$.
Then it follows that $\tau^{dn} 1_Y \in Y_0$ for all $n\in\N$.
	Since the family $\{g_1, g_2, g_3\}$ is not of the types $(e_1), (e_2)$, $(e_3)$, the family $\{h_1, h_2, h_3\}$ with $h_i(n) = g_i(dn)$ is also not of these types.
Hence by \cite[Theorem C]{Frantzikinakis08},  the set $S = \{n \in \mathbb{N}\colon \phi(dn) > \mu(A)^4 - \epsilon/4\}$ is syndetic. Together with (\ref{equation:delta-is-null}), we get
	\[
	\lim_{N \to \infty} \frac{1}{|[N]\cap S|} \sum_{n \in [N] \cap S} |\delta(dn)| = 0,
	\]
	which implies
	\[
	\limsup_{N \to \infty} \frac{1}{|[N] \cap S|} \sum_{n \in [N] \cap S} |\phi(dn) - F(\tau^{dn} 1_Y)| < \epsilon/4.
	\]
	We deduce that there exists an $n \in \N$ such that $F(\tau^{dn} 1_Y) > \mu(A)^4 - \epsilon/2$.
	
	Since $\tau^{dn}1_Y \in Y_0$ and $F$ is continuous, there is an open subset $U$ of $Y_0$ such that $F > \mu(A)^4 - 3 \epsilon/4$ on $U$.
	By \cite[Corollary 1.4]{Le17}, the sequence $(\tau^{p_n -1} 1_Y)_{n \in \N}$ is equidistributed on $Y_0$ when restricted to $p_n \equiv 1 \mod d$.
	Hence the set $R:=\{n\in\N\colon \tau^{p_n-1} 1_Y \in U\}$ has positive density, and for every $n\in R$ we have $F(\tau^{p_n -1} 1_Y) > \mu(A)^4 - 3\epsilon/4$.
	On the other hand, from \eqref{equation:null-along-prime-minus-1} it follows that the set $R':=\{n\in R\colon \phi(p_n-1)<\mu(A)^4-\epsilon\}$ has $0$ density. Therefore the set $R\setminus R'$ has positive density and is contained in the set \eqref{equation:glob-prime-3}. This finishes our proof.
\end{proof}

\cref{thm_primesOR} now follows from \cref{theorem:glob-primes} by letting $g_2(x) = 2 g_1(x)$, $g_3(x) = 3 g_1(x)$ and
\[
    g_1(x) = \begin{cases}
        f(x + 1) \mbox{ if } f(1) = 0 \\
        f(x - 1) \mbox{ if } f(-1) = 0
    \end{cases}
\]
for all $x \in \Z$.

\medbreak

\subsection{Optimal recurrence along Hardy field sequences}\label{sec_hardy}

We prove a slight generalization of \cref{theorem:glob-Hardy}.

\begin{theorem}\label{theorem:glob-Hardyfull}
    Let $f \in\mathcal{H}$ polynomial growth and satisfying $\big|f(x)-cp(x)\big|/\log(x)\to\infty$ for every $c\in\R$, $p\in\Z[x]$.
	Let $(X,{\mathcal B},\mu,T)$ be an ergodic measure preserving system, $A \in{\mathcal B}$ and $\epsilon>0$. Let $0 \leq a_1 \leq a_2 \leq a_3 \in \mathbb{Z}$.
    Then the sets
    \begin{equation*}
	\{n \in \mathbb{N}\colon \mu(A \cap T^{-a_1 \lfloor f(n) \rfloor} A) > \mu(A)^2 - \epsilon\}
	\end{equation*}
	and
	\begin{equation*}
	\{n \in \mathbb{N}\colon  \mu(A \cap T^{-a_1 \lfloor f(n) \rfloor} A \cap T^{-a_2 \lfloor f(n) \rfloor} A) > \mu(A)^3 - \epsilon\}
	\end{equation*}
	have positive lower density. If $a_3 = a_1 + a_2$ then the set
	\begin{equation}
	\label{eq:july-15-1}
	\{n \in \mathbb{N}\colon \mu(A \cap T^{-a_1 \lfloor f(n) \rfloor} A \cap T^{-a_2 \lfloor f(n) \rfloor} A \cap T^{-a_3 \lfloor f(n) \rfloor} A) > \mu(A)^4 - \epsilon\}
	\end{equation}
	also has positive lower density.
\end{theorem}

\begin{proof}
	We again only prove the set in \eqref{eq:july-15-1} has positive density since the proof for other two sets is the same. Similar to  the proof of \cref{theorem:glob-primes}, let
	\[
	    \phi(n) = \mu(A \cap T^{-a_1 n} A \cap T^{-a_2 n} A \cap T^{-a_3 n} A).
	\]
	Write $\phi(n) = \psi(n) + \delta(n)$, where $(\psi(n))_{n \in \N}$ is a nilsequence and $(\delta(n))_{n \in \N}$ satisfies
	\[
	    \lim_{N - M \to \infty} \frac{1}{N-M} \sum_{n=M}^{N-1} |\delta(n)| = 0.
	\]
	Let $(F(\tau^n 1_Y))_{n \in \N}$ be an approximation of $(\psi(n))_{n \in \N}$ as in the proof of \cref{theorem:glob-primes}.
By \cite[Theorem 1.1]{Le17},
	\[
	\lim_{N \to \infty} \frac{1}{N} \sum_{n=1}^N \Big|\delta\Big(\big\lfloor f(n) \big\rfloor\Big)\Big| = 0.
	\]
This implies
\begin{equation}
\label{eq:glob-Hardy-8.30.1}
    \limsup_{N \to \infty} \frac{1}{N} \sum_{n = 1}^N \bigg|\phi\Big(\big\lfloor f(n) \big\rfloor\Big) - F\Big(\tau^{\lfloor f(n) \rfloor} 1_Y\Big)\bigg| < \frac\epsilon4.
\end{equation}

	As in the proof of \cref{theorem:glob-primes}, there is an open set $U \subset Y$ such that $F > \mu(A)^4 - 3\epsilon/4$ on $U$.
	By \cite[Theorem 1.2]{Frantzikinakis09}, the sequence $(\tau^{\lfloor a(n) \rfloor} 1_X)_{n \in \N}$ is equidistributed on $Y$.
	Hence the set $\{n \in \mathbb{N}\colon F(\tau^{\lfloor a(n) \rfloor} 1_Y) > \mu(A)^4 - 3 \epsilon/4\}$ has positive density.
	This fact combined with (\ref{eq:glob-Hardy-8.30.1}) implies the set $\{n \in \mathbb{N}\colon \phi(\lfloor a(n) \rfloor) > \mu(A)^4 - \epsilon\}$ has positive lower density.
\end{proof}

\begin{remark}
	In the proof above, we do not utilize the fact that the open set $U$ is inside the identity component $Y_0$ of $Y$. This is because the orbit along $(\lfloor f(n) \rfloor)_{n \in \N}$ is equidistributed on the entire $Y$. On the contrary, the orbit along primes minus $1$ is only equidistributed on some connected components of $Y$ where $Y_0$ is one of them.
\end{remark}

\subsection{Optimal recurrence along Beatty sequences}
\label{secbeatty_sequences}
Let $(X, \mathcal{B}, \mu, T)$ be a measure preserving system.
The \emph{discrete spectrum} $\sigma(T)$ of $T$ is the set of eigenvalues $\theta \in \mathbb{T} := \mathbb{R}/\mathbb{Z}$ for which there exists a non-zero eigenfunction $f \in L^{2}(\mu)$ satisfying $f(T x)= e^{2 \pi i \theta} f(x)$ for $\mu$-almost every $x \in X$.

Given a measure preserving system $(X, \mathcal{B}, \mu, T)$, the transformation $T^q$ is ergodic if and only if $\sigma(T) \cap \langle 1/q \rangle = \{0\}$, where $\langle 1/q \rangle$ denotes the abelian group generated by $1/q$, as we view $1/q$ as an element of the group $\T$. Therefore
\cref{thm_progressionOR} follows from the next result.

\begin{theorem}\label{thm_beatty}
	Let $\theta, \gamma \in \mathbb{R}$ with $\theta > 0$ and $(X, \mathcal{B}, \mu, T)$ be an ergodic system whose discrete spectrum $\sigma(T)$ satisfies $\sigma(T) \cap \langle 1/\theta \rangle = \{0\}$. Let $0 \leq a_1 \leq a_2 \leq a_3 \in \mathbb{Z}$. Then for any $A \in \mathcal{B}$ and $\epsilon > 0$, the sets
	\[
	\{n \in \mathbb{N}\colon \mu(A \cap T^{-a_1 \lfloor \theta n + \gamma \rfloor} A) > \mu(A)^2 - \epsilon \}
	\]
	and
	\[
	\{n \in \mathbb{N}\colon \mu(A \cap T^{-a_1 \lfloor \theta n + \gamma \rfloor} A \cap T^{-a_2 \lfloor \theta n + \gamma \rfloor} A) > \mu(A)^3 - \epsilon \}
	\]
	are syndetic. If $a_3 = a_1 + a_2$ then the set
	\[
	\{n \in \mathbb{N}\colon \mu(A \cap T^{-a_1 \lfloor \theta n + \gamma \rfloor} A \cap T^{-a_2 \lfloor \theta n + \gamma \rfloor} A \cap T^{-a_3 \lfloor \theta n + \gamma \rfloor} A) > \mu(A)^4 - \epsilon\}
	\]
	is also syndetic.
\end{theorem}

\begin{proof}
	If $0 < \theta \leq 1$, then the set $S = \{\lfloor \theta n + \gamma \rfloor\colon n \in \mathbb{N} \}$ contains all but finitely many elements of $\N$. Hence the conclusion follows trivially from \cite[Theorem 1.2]{Bergelson_Host_Kra05}.
	
	Assume $\theta > 1$. Define $\phi, \psi, \delta, Y, F(\tau^n 1_Y)$ as in the proof of \cref{theorem:glob-Hardyfull}. Then by \cite[Theorem 2.1]{Moreira_Richter19}, the discrete spectrum of $(Y, \mu_Y, \tau)$ is contained in the discrete spectrum of $(X, \mu, T)$. Hence $\sigma(\tau) \cap \langle 1/\theta \rangle = \{0\}$.
	\begin{claim}
	\label{claim:july-15-1}
		The sequence $(\tau^{\lfloor \theta n + \gamma\rfloor} 1_Y)_{n \in \mathbb{N}}$ is well distributed on $Y$.
	\end{claim}
	\begin{proof}
		It suffices to show that for every $F \in C(Y)$,
		\begin{equation}
		\label{equation:beatty-1}
		\lim_{N - M \to \infty} \frac{1}{N-M} \sum_{n=M}^{N-1} F(\tau^{\lfloor \theta n + \gamma\rfloor} 1_Y) = \int_Y F \, d \mu_Y.
		\end{equation}
		As before, let $S = \{\lfloor \theta n + \gamma\rfloor\colon n \in \mathbb{N}\}$. Then  $m \in S$ if and only if $m \in \Z$ and
		\[
		\theta n + \gamma - 1 < m \leq \theta n + \gamma,
		\]
		or equivalently,
		\[
		n - \frac{1- \gamma}{\theta} < m \theta^{-1} \leq n + \frac{\gamma}{\theta}
		\]
		for some $n \in \Z$.
		This is equivalent of saying $ m \theta^{-1} \mod 1 \in J$, where $J = [0, \gamma/\theta] \cup ((1- \gamma)/\theta, 1)$.
		
		Let $W = \overline{ \{n \theta^{-1} \mod 1\colon n \in \mathbb{N}\}}$, where $\overline{\{ \cdot \}}$ denotes taking closure in $\T$.
		Then $W$ is a closed subgroup of $ \mathbb{T}$. Since $\sigma(\tau) \cap \langle \theta^{-1} \rangle = \{0\}$, for any $F \in C(Y)$ and $G \in C(W)$,
		\[
	    	\lim_{N - M \to \infty} \frac{1}{N-M} \sum_{m=M}^{N-1} F(\tau^m 1_Y) G(m \theta^{-1}) = \int_{Y} F \, d \mu_Y \int_{W} G \, d \mu_W.
		\]
		This follows from the fact that two nilsystems are disjoint if and only if their discrete spectrums are disjoint.
		Approximating the Riemann integrable function $\one_{J \cap W}$ by continuous functions, we then get
		\[
		\lim_{N - M \to \infty} \frac{1}{N-M} \sum_{m=M}^{N-1} F(\tau^m 1_Y) \one_{J \cap W} (m \theta^{-1}) = \mu_W(J \cap W) \int_Y F \, d \mu_Y,
		\]
		or equivalently
		\begin{equation}
		\label{equation:beatty-2}
		\frac{1}{\mu_W(J \cap W)} \lim_{N - M \to \infty} \frac{1}{N-M} \sum_{m=M}^{N-1} F(\tau^m 1_Y) \one_{J \cap W} (m \theta^{-1}) = \int_Y F \, d \mu_Y.
		\end{equation}
		Note that $\{m \theta^{-1}\} \in J \cap W$ if and only if $m \in S$, and the uniform density of $S$ is exactly $\mu_W(J \cap W)$. Therefore the left hand side of (\ref{equation:beatty-2}) is the same as the left hand side of (\ref{equation:beatty-1}). This proves \cref{claim:july-15-1}.
\end{proof}
	
	Now we return to the proof of \cref{thm_beatty}. As shown in the proof of \cref{theorem:glob-primes}, there exists an open set $U \subset Y$ such that $F > \mu(A)^{4} - 3 \epsilon/4$ on $U$. Since $(\tau^{\lfloor \theta n + \gamma \rfloor} 1_X)_{n\in\mathbb{N}}$ is well distributed on $Y$, the set  $S = \{n \in \mathbb{N}\colon F(\tau^{\lfloor \theta n + \gamma \rfloor} 1_X) > \mu(A)^4 - 3 \epsilon/4\}$ is syndetic.
	Since the sequence $(\delta(n))_{n\in\N}$ tends  to zero in the uniform density, and the set $\{\lfloor \theta n + \gamma\rfloor\colon n \in S\}$ has positive uniform density, we have
	\[
	\lim_{N - M \to \infty} \frac{1}{|S \cap [M, N)|} \sum_{n \in S \cap [M, N)} |\delta(\lfloor \theta n + \gamma \rfloor)| = 0.
	\]
	Therefore
	\[
	\limsup_{N - M \to \infty} \frac{1}{|S \cap [M, N)|} \sum_{n \in S \cap [M, N)} |\phi(\lfloor \theta n + \gamma\rfloor) - F(\tau^{\lfloor \theta n + \gamma \rfloor} 1_Y)| < \epsilon/4.
	\]
	Since $F(\tau^{\lfloor \theta n + \gamma\rfloor} 1_Y) > \mu(A)^4 - 3 \epsilon/4$ for $n \in S$, we get that the set of $n \in S$ such that $\phi(\lfloor \theta n + \gamma\rfloor) > \mu(A)^4 - \epsilon$ is syndetic. This finishes the proof.
\end{proof}

\section{Optimal recurrence along \texorpdfstring{$\{a_1 n, a_2 n, \ldots, a_k n\}$}{(a1 n, a2 n,..., ak n)}}
\label{s4}

\subsection{Lack of optimal recurrence for \texorpdfstring{$\mathbf{k=5}$}{k=5}}

\label{sec_k=5}

In this section, we prove \cref{thm_BadRecurrence}.
We adapt the proof of Theorem 1.3 in \cite{Bergelson_Host_Kra05}.

We use the measure preserving system $(X,\mathcal{B},\mu,T)$, where $X=\T^2$ is the $2$-dimensional torus, $\mu$ is the Haar measure, and $T(x,y)=(x+\alpha,y+2x+\alpha)$ for some irrational $\alpha\in\R$.
It is well known that this system is totally ergodic.
For every $n\in\Z$ and every point $(x,y)\in\T^2$, a quick computation shows that $T^n(x,y)=(x+n\alpha,y+2nx+n^2\alpha)$.

Let $\ell>1$. We take a suitably large $L,C\in\N$ and a set $\Lambda\subset\{0,\dots,L-1\}$ to be chosen later.
Let
$$B:=\bigcup_{b\in\Lambda}I_b\quad\text{where }I_b:=\left[\frac b{CL},\frac b{CL}+\frac1{C^2L}\right)$$
and let $A=\T\times B$.
For each $n\in\Z$, in order for a point $(x,y)$ to belong to $T^{-a_1n}A\cap\cdots\cap T^{-a_5n}A$, we need $y_i:=y+2a_inx+a_i^2n^2\alpha\in B$ for each $i=1,\dots,5$.
Let $b_i\in\Lambda$ be such that $y_i\in I_{b_i}$.

We now need the following elementary lemma.
\begin{lemma}\label{lemma_elementary}
Let $a_1,\dots,a_4\in\Z$ be distinct and let $M$ be the $4\times 3$ matrix whose $(i,j)$ entry is $a_i^j$ for $i=1,\dots,4$ and $j=0,1,2$.
For each $i=1,\dots,4$, let $v_i$ be $(-1)^i$ times the determinant of the matrix obtained from $M$ by deleting the $i$-th row.
Then for every quadratic polynomial $f\in\R[x]$, $$v_1f(a_1)+v_2f(a_2)+v_3f(a_3)+v_4f(a_4)=0.$$
\end{lemma}

\begin{proof}
The claim amounts to the statement that the matrix
$$\left[\begin{array}{cccc}
f(a_1) & 1& a_1 & a_1^2\\
f(a_2) & 1& a_2 & a_2^2\\
f(a_3) & 1& a_3 & a_3^2\\
f(a_4) & 1& a_4 & a_4^2
\end{array}\right]$$
has determinant $0$.
But this follows from the fact that any quadratic polynomial is a linear combination of the polynomials $1,x,x^2$.
\end{proof}

In view of \cref{lemma_elementary}, there exist integers $v_1,\dots,v_4$ and $\tilde v_2,\dots\tilde v_5$ (depending only on $a_1,\dots,a_5$, but not on $x$ nor $y$) such that $v_1y_1+\cdots+v_4y_4=0$ and $\tilde v_2y_2+\cdots+\tilde v_5y_5=0$.
Therefore, if $C$ is large enough, then also $v_1b_1+\cdots+v_4b_4=\tilde v_2b_2+\cdots+\tilde v_5b_5=0$, as it will be an integer which can be made smaller than $1$ when $C$ is large enough.

Suppose now that $\Lambda$ does not contain any solution to $v_1b_1+\cdots+v_4b_4=\tilde v_2b_2+\cdots+\tilde v_5b_5=0$ except when $b_1=\cdots=b_5$.
Then, if $(x,y)\in T^{-a_1n}A\cap\cdots\cap T^{-a_5n}A$, all the $y_i$ must belong to the same $I_b$, which implies that $x\in X_n$, where $X_n$ is the set of points $x\in\T$ satisfying $\big\|2n(a_2-a_1)x\big\|_\T<1/C^2L$.
Since $y_1\in B$, the point $y$ must belong to the set $B-2a_1nx-a_1^2n^2\alpha$, which being a shift of $B$ has the same measure as $B$.
We conclude that
$$\mu(T^{-a_1n}A\cap\cdots\cap T^{-a_5n}A)\leq \mu_\T(X_n)\mu_\T(B)=\frac2{C^4L^2}|\Lambda|.$$

Since $\mu(A)=|\Lambda|\frac1{C^2L}$, a quick computation now shows that the proof will be complete once we construct a set $\Lambda\subset\{0,\dots,L-1\}$ with $|\Lambda|>L^{1-1/\ell}$ and without non-constant solutions to $v_1b_1+\cdots+v_4b_4=\tilde v_2b_2+\cdots+\tilde v_5b_5=0$.
The existence of such a set $\Lambda$ is provided by the following lemma.

\begin{lemma}
\label{lemma_Ruzsa}
Let $a_1,\dots,a_5\in\Z$ be pairwise distinct and let $v_i$ and $\tilde v_i$ be described in the paragraph after \cref{lemma_elementary}.
For every $\epsilon>0$ and every large enough $L\in\N$, there exists a set $\Lambda\subset\{0,\dots,L-1\}$ with $|\Lambda|>L^{1-\epsilon}$  such that the only $b_1,\dots,b_5\in\Lambda$ satisfying $v_1b_1+\cdots+v_4b_4=\tilde v_2b_2+\cdots+\tilde v_5b_5=0$ also satisfy $b_1=\cdots=b_5$.
\end{lemma}

\cref{lemma_Ruzsa} is a generalization of \cite[Theorem 2.4]{Bergelson_Host_Kra05}, due to Ruzsa, corresponding to $a_i=i$.
The key to proving \cref{lemma_Ruzsa} is the following intermediate result.

\begin{lemma}\label{lemma_samenorm}
Let $a_1<\dots<a_5$ be integers and let $v_i$ and $\tilde v_i$ be as described above.
Let $d\in\N$ and let $b_1,\dots,b_5\in\R^d$ all have the same Euclidean norm.
If $v_1b_1+\cdots+v_4b_4=\tilde v_2b_2+\cdots+\tilde v_5b_5=0$ then $b_1=\cdots=b_5$.
\end{lemma}

Unfortunately \cref{lemma_samenorm} does not hold for arbitrary $v_i$ and $\tilde v_i$, as seen by the example $v_1=v_3=\tilde v_3=\tilde v_5=1$ and $v_2=v_4=\tilde v_2=\tilde v_4=-1$ which would provide a counterexample with $d=1$ and $b_1=b_2=b_5=1$ and $b_3=b_4=-1$.
Indeed we will need to use the description of the $v_i$ and $\tilde v_i$ given by \cref{lemma_elementary} and this makes the proof somewhat cumbersome.

\begin{proof}
The condition $a_1<\cdots<a_5$ implies that $v_1,v_3>0$ and $v_2,v_4<0$.
Let
$$S:=\frac{v_1b_1+v_3b_3}{v_1+v_3},\qquad A:=b_1-S,\qquad B=b_2-S.$$
Applying \cref{lemma_elementary} to a constant polynomial, we get that $v_1+v_2+v_3+v_4=0$ and hence, together with $v_1b_1+\cdots+v_4b_4=0$, that $S=(v_2b_2+v_4b_4)/(v_2+v_4)$.
Then we have
$$b_1=S+A,\qquad b_2=S+B,\qquad b_3=S-\frac{v_1}{v_3}A,\qquad b_4=S-\frac{v_2}{v_4}B.$$
Our goal is to show that $b_1=b_2=b_3=b_4=S$, and so it suffices to show that $A=B=0$ (the fact that also $b_5=S$ would then immediately follow from the equation $\tilde v_2b_2+\cdots+\tilde v_5b_5=0$).
Since the quantity $\|b_i\|^2-\|S\|^2$ does not depend on $i$, we find that the following $4$ numbers are equal
\begin{equation}\label{eq_equal}
\|A\|^2+2\langle S,A\rangle,
\qquad
\|B\|^2+2\langle S,B\rangle,
\qquad
\frac{v_1^2}{v_3^2}\|A\|^2-\frac{2v_1}{v_3}\langle S,A\rangle,
\qquad
\frac{v_2^2}{v_4^2}\|B\|^2-\frac{2v_2}{v_4}\langle S,B\rangle.
\end{equation}
Equality between the first and third gives $2\langle S,A\rangle= \|A\|^2\left(\tfrac{v_1}{v_3}-1\right)$; equality between the second and fourth gives $2\langle S,B\rangle=\|B\|^2\left(\tfrac{v_2}{v_4}-1\right)$ and then equality between the first two numbers implies
\begin{equation}\label{eq_normequality}
  \|A\|^2v_1v_4=\|B\|^2v_2v_3.
\end{equation}
In order to show that $A=B=0$, we first show that $B$ is a positive scalar multiple of $A$.
Once we do that, we have from \eqref{eq_normequality} that $B=\sqrt{\tfrac{v_1v_4}{v_2v_3}}A$ and hence, equality between the first and last quantities from \eqref{eq_equal} (together with $2\langle S,A\rangle= \|A\|^2\left(\tfrac{v_1}{v_3}-1\right)$) gives
$$\|A\|^2\frac{v_1}{v_3}=\frac{v_1v_2}{v_3v_4}\|A\|^2- 2\sqrt{\frac{v_1v_2}{v_3v_4}}\langle S,A\rangle\quad\iff\quad\|A\|^2\sqrt{\frac{v_1}{v_3}} \left(\sqrt{\frac{v_1}{v_3}}-\sqrt{\frac{v_2}{v_4}}\right) \left(1+\sqrt{\frac{v_1v_2}{v_3v_4}}\right)=0.$$
This implies that $A=0$ unless $\tfrac{v_1}{v_3}=\tfrac{v_2}{v_4}$.
Using the description of each $v_i$ from \cref{lemma_elementary} as a Vandermonde determinant, this is equivalent to $(a_1-a_4)^2=(a_2-a_3)^2$. Since we are assuming that $a_1<a_2<a_3<a_4$ this can not happen and hence $A=0$.

We have reduced the proof to showing that $B$ is a positive scalar multiple of $A$.
It is now the time to use the fact that also $\tilde v_2b_2+\cdots+\tilde v_5b_5=0$.
From \cref{lemma_elementary}, we deduce that $\tilde v_5=v_1$, and so we can write $b_5$ in terms of $S,A,B$ as
$$b_5=
\frac1{v_1}(-\tilde v_2b_2-\tilde v_3b_3-\tilde v_4b_4)
=
S+\frac{\tilde v_3}{v_3}A
+\left(\frac{v_2\tilde v_4-\tilde v_2v_4}{v_4v_1}\right)B=S+\alpha A+\beta B,$$
where $\alpha:=\frac{\tilde v_3}{v_3}$ and $\beta:=\frac{v_2\tilde v_4-\tilde v_2v_4}{v_4v_1}$.
Using the relations established above to write $\|B\|^2$, $\langle S,A\rangle $ and $\langle S,B\rangle$ in terms of $\|A\|^2$ we compute
$$\|b_5\|^2-\|S\|^2=2\alpha\beta\langle A,B\rangle+
\|A\|^2\left[\alpha^2+\alpha\left(\frac{v_1}{v_3}-1\right)+\left(\beta^2+\beta\left(\frac{v_2}{v_4}-1\right)\right)\frac{v_1v_4}{v_3v_2}\right].$$
Since $\|b_5\|=\|b_1\|$, we deduce that $\|b_5\|^2-\|S\|^2=\|A\|^2\tfrac{v_1}{v_3}$. After a somewhat tedious computation, we eventually arrive at $\langle A,B\rangle=\|A\|^2\sqrt{\tfrac{v_1v_2}{v_3v_4}}=\|A\|\cdot\|B\|$.
But this implies that $B$ must be a positive scalar multiple of $A$ as desired, finishing the proof.
\end{proof}

\begin{proof}[Proof of \cref{lemma_Ruzsa}]
     Let $C=|v_1|+\cdots+|v_4|+|\tilde v_2|+\cdots+|\tilde v_5|$, let $d>2/\epsilon$ be a natural number and then let $m\in\N$ be large enough multiple of $C$ depending only on $C,d$ and $\epsilon$ (in fact, we need that $m^{d\epsilon-2}>C^{d-2}d)$. Set $L=m^d$.
     We can express each number in $\{0,\dots,L-1\}$ using $d$ digits in base $m$ expansion.
     Let
    \[
        F := \left\{x_0 + x_1 m + \ldots + x_{d-1} m^{d-1}\colon x_i \in \left[0,\dots,\frac mC\right)\right\}.
    \]
    We have $|F|=(m/C)^d$. Let $r\colon F\to\N$ be the sum of the squares of the digits in base $m$, in other words, $r(x_0 + x_1 m + \ldots + x_{d-1} m^{d-1})=x_0^2+\cdots+x_{d-1}^2$.
    Then $r(F)\subset[0,dC^2/m^2)$.
    Therefore there exists $r_0\in[0,dC^2/m^2)$ such that the set
    $$\Lambda:=\{x\in F\colon r(x)=r_0\}$$
    has cardinality $|\Lambda|\geq(m/C)^{d-2}/d$.
    The choice of parameters above yields $|\Lambda|>L^{1-\epsilon}$.

    Finally, suppose that $b_1,\dots,b_5\in\Lambda$ satisfy $v_1b_1+\cdots+v_4b_4=\tilde v_2b_2+\cdots+\tilde v_5b_5=0$.
    We identify each $b_i$ with the vector in $\R^d$ obtained from its digits in base $m$.
    Then $\|b_1\|=\cdots=\|b_5\|$.
    Since each digit in $b_i$ is at most $m/C$, there is no carryover when multiplying by $v_i$ or $\tilde v_i$ and thus, the equations $v_1b_1+\cdots+v_4b_4=\tilde v_2b_2+\cdots+\tilde v_5b_5=0$ apply even when multiplication and addition is being performed in $\R^d$.
    Applying \cref{lemma_samenorm}, we conclude that indeed $b_1=\cdots=b_5$ as desired.
\end{proof}

\subsection{Abundance of solutions implies optimal recurrence for \texorpdfstring{$\mathbf{k = 4}$}{k=4}}

\label{sec_proofintegeranalysis}

In this section we prove Theorem \ref{thm:integerdirect}.
We first need to reformulate the assumptions in terms of functions on a torus; this is the content of \cref{lem:equ} below.
We start with an estimate from harmonic analysis.
Let $A$ be a finite set,  $f\colon A\to\R$ a function and $p>0$.
We denote by $\|f\|_{L^p}$ its usual $L^p$ quasinorm when $A$ is endowed with the normalized counting probability measure, i.e.
$$\|f\|_{L^p}:=\left(\frac1{|A|}\sum_{a\in A}\big|f(a)\big|^p\right)^{1/p}.$$
We will also make use of the weak $L^p$ quasinorm:
$$\Vert f\Vert_{L^p_w}:=\sup_{s>0}~s\cdot\left(\frac{\vert\{a\in A\colon\vert f(a)\vert>s\}\vert}{|A|}\right)^{1/p}.$$
We remark that when $p<1$ these quasinorms do not satisfy the triangle inequality.
We will only use these quasinorms with $p<1$ to invoke the following well known interpolation lemma. We include its short proof for completeness.
\begin{lemma}\label{thm:itp}
	Let $0<p<r<\infty$ and let $A$ be a finite set. For every function $f\colon A\to\R$ we have
	$$\Vert f\Vert^{r}_{L^r}\leq \frac{r}{r-p}\Vert f\Vert^{p}_{L^p_w}\Vert f\Vert^{r-p}_{L^{\infty}}.$$
\end{lemma}
\begin{proof}
	Combining the identity
	$$x^r=\int_0^xrs^{r-1}\,ds=r\int_0^\infty s^{r-1}\one_{[0,x]}(s)\,ds =r\int_0^\infty s^{r-1}\one_{\{x>s\}}\,ds$$
	with the definition of $L^r$ norm, we deduce the formula
	$$\|f\|^r_{L^r}
	=
	r\int_0^\infty s^{r-1}\frac1{|A|}\sum_{a\in A}\one_{\{|f(a)|>s\}}\,ds
	=
	r\int_0^\infty s^{r-1}\frac{\big|\{a\in A\colon |f(a)|>s\}\big|}{|A|}\,ds.$$
	Finally, using the definition of the weak $L^p$ norm, we conclude
	$$\|f\|^r_{L^r}
	\leq
	r\int_0^\infty s^{r-1-p}\|f\|^p_{L^p_w}\,ds
	=
	\frac r{r-p}\Vert f\Vert^{p}_{L^p_w}\Vert f\Vert^{r-p}_{L^{\infty}}.$$
\end{proof}

The following lemma makes use of the quantities $d_{m,N}(E)$ and $D_{m,N}(V,E)$ introduced in \cref{definition_densities}.

\begin{lemma}[Equivalent inequalities]\label{lem:equ}
	Let $m,d,\ell\in\N$ with $\ell>d$, let $C>0$, $V\subseteq\mathbb{Q}^{d}$ be a subspace containing the vector $(1,\dots,1)$ and $\overline{V}\in\mathbb{R}^{d}$ be its closure in $\R^{d}$.
	Then $(1)\Leftrightarrow (2)\Rightarrow (3)\Rightarrow (4)$:	
	\begin{enumerate}
		\item For every large enough $N$ and every subset $E\subseteq [N]^{m}$, we have $D_{m,N}(V^{m},E)\geq C d_{m,N}(E)^{\ell}$.
		\item For every large enough $N$ and every function $c\colon[N]^{m}\to[0,1]$, we have
		$$\frac{1}{\vert V\cap [N]^{d}\vert^{m}}\sum_{a_{i}\in [N]^{m}, (a_{1},\dots,a_{d})\in V^m}c(a_{1})c(a_{2})\dots c(a_{d})\geq C\Vert c\Vert_{L^{d/\ell}_w}^{d}.$$
		\item For every large enough $N$ and every function $c\colon[N]^{m}\to[0,1]$, we have
		$$\frac{1}{\vert V\cap [N]^{d}\vert^{m}}\sum_{a_{i}\in [N]^{m}, (a_{1},\dots,a_{d})\in V^m}c(a_{1})c(a_{2})\dots c(a_{d})\geq C\left(1-\frac{d}{\ell}\right)^{\ell}\Vert c\Vert_{L^{1}}^{\ell}.$$
		\item  Let $Y=\big(\overline{V}/\mathbb{Z}^{d}\big)^{m}$ be a subtorus of $\mathbb{T}^{d\times m}$. For every measurable function $f\colon \mathbb{T}^m\to[0,1]$,
		$$\int_{Y}f(y_{1})f(y_{2})\dots f(y_{d})\,d\mu_{Y}(y_{1},\dots,y_{d})
		\geq C\left(1-\frac{d}{\ell}\right)^{\ell}
		\left(\int_{\mathbb{T}^m}f\,d\mu_{\mathbb{T}^m}\right)^{\ell}.$$
	\end{enumerate}	
\end{lemma}	

\begin{remark*}
	\begin{itemize}
		\item
		Whenever we have a point $x$ in $[N]^{dm}$ (or analogously for $\Q^{dm}$, $\T^{dm}$, etc.) we consider $x=(x_{i,j})_{i=1,\dots,m,\,j=1,\dots,d}$ with each $x_{i,j}\in[N]$.
		We then write $x=(x_1,\dots,x_d)$ where each $x_i\in[N]^m$ is the vector $x_i=(x_{i,j})_{j=1}^m$.
		Depending on the context, we may also write $x=(x_1,\dots,x_m)$, where now each $x_j\in[N]^d$ is the vector $x_j=(x_{i,j})_{i=1}^d$ (it should be clear at any point which vectors we are referring to).
		
		For instance, if $ v=(v_{i,j})_{i=1,\dots,m,\,j=1,\dots,d}\in\Q^{dm}$, then $v$ is in $V^m$ if for every $i$ the vector $(v_{i,j})_{j=1}^d$ of $\Q^d$ belongs to $V$; and $v$ is in $E^d$ if for every $j$ the vector $(v_{i,j})_{i=1}^m$ is in $E$.
		Similarly in (2) and (3), the statement that $(a_1,\dots,a_d)\in V^m$ should be interpreted by writing each $a_j$ as $(a_{i,j})_{i=1}^m$ and requiring that each vector $(a_{i,j})_{j=1}^d$ is in $V$.
		\item It might be true that (3) and (4) are also equivalent to (1) and (2), but we didn't find a proof, and that direction is not needed in this paper.
	\end{itemize}
\end{remark*}

\begin{proof}
	(2)$\Rightarrow$ (1). Take $c(a)=\one_{E}(a)$.
	
	(1)$\Rightarrow$ (2). Let $p:=d/\ell$, observe that $\|c\|_{L^{p}_w}=\frac1{N^{m\ell/d}}\sup_{s\geq0}s\vert\{a\in[N]^{m}\colon c(a)>s\}\vert^{\ell/d}$ and assume that the maximum is obtained at $s=t$.
	Let $E=\{a\in[N]^{m}\colon c(a)>t\}$.
	Since $c\geq t\one_{E}$, we have
	$$\frac{1}{\vert V\cap [N]^{d}\vert^{m}}\sum_{a_{i}\in [N]^{m}, (a_{1},\dots,a_{d})\in V}c(a_{1})c(a_{2})\dots c(a_{d})
	\geq
	t^{d}D_{m,N}(V^{m},E).$$
	Invoking (1), we get
	$t^{d}D_{m,N}(V^{m},E)\geq C\frac{t^{d}\vert E\vert^{\ell}}{N^{\ell m}}=C\Vert c\Vert^{d}_{L^{p}_w}.$
	
	(2)$\Rightarrow$ (3). We only need to show (3) for $c\neq 0$.
	By \cref{thm:itp},
	$$\Vert c\Vert^{d}_{L^{p}_w}\geq
	\left(1-\frac{d}{\ell}\right)^{\ell}\frac{\Vert c\Vert_{L^{1}}^{\ell}}{\Vert c\Vert_{L^{\infty}}^{\ell-d}}
	\geq
	\left(1-\frac{d}{\ell}\right)^{\ell}\Vert c\Vert_{L^{1}}^{\ell}.
	$$

	(3)$\Rightarrow$(4).
	Let
	$$Y_N=\bigcup_{a\in(V\cap[N]^d)^m}\prod_{i=1}^d\prod_{j=1}^m\left[\frac{a_{i,j}-1}N,\frac{a_{i,j}}N\right)\subset\T^{dm}$$
	and let $\mu_N$ be the normalized probability measure supported on $Y_N$.
	We claim that $\mu_{N!}\to\mu_Y$ as $N\to\infty$.
	Indeed, any limit point of the sequence $(\mu_N)_{N\in\N}$ must be supported on $Y$.
	Moreover, for any $v\in(V/\Z^d)^m$, if $N$ is larger than the denominators of all coordinates of $v$, then $\mu_{N!}$ is invariant under $v$.
	We conclude that any limit point of the sequence $(\mu_{N!})_{N\in\N}$ is supported on $Y$ and invariant under $Y$, hence it must be $\mu_Y$.
	
	Now given $f\colon\T^m\to[0,1]$, let $c\colon [N]^m\to[0,1]$ be the function
	$$c(a)=N^m\int_{[-1/N,0)^m}f(a/N+x)dx.$$
	When $N$ is large enough, we have that
	\begin{eqnarray*}&
		\displaystyle\int_{Y_N}\prod_{i=1}^df(y_i)\,d\mu_N(y_1,\dots,y_d)
		=
		\frac{N^{dm}}{|V\cap[N]^d|^m}\sum_{a\in(V\cap[N]^d)^m}\int_{\left[\tfrac{-1}N,0\right)^{dm}}\prod_{i=1}^df\left(\frac{a_i}N+x_i\right)d(x_1,\dots,x_d)
		\\ & \geq\displaystyle
		\frac{1}{|V\cap[N]^d|^m}\sum_{a_{i}\in [N]^{m}, (a_{1},\dots,a_{d})\in V^m}c(a_{1})c(a_{2})\dots c(a_{d})\\&\geq C\left(1-\frac{d}{\ell}\right)^{\ell}\Vert c\Vert_{L^{1}}^{\ell}
		=C\left(1-\frac{d}{\ell}\right)^{\ell}\left(\int_{\T^m} f\,d\mu_{\T^m}\right)^\ell.
	\end{eqnarray*}
\end{proof}
\medbreak

We are now ready to prove Theorem \ref{thm:integerdirect}.
\begin{proof}[Proof of Theorem \ref{thm:integerdirect}.]
Fix an ergodic system  $(X,\mathcal{B},\mu,T)$ and a set $A\in\mathcal{B}$ with $\mu(A)>0$.
    Let $Z_{2}$ be the 2-step nilfactor of $X$.
	Using a standard approximation argument, we can assume that $Z_{2}$ is a 2-step nilsystem $(G/\Gamma,\mu_{Z_2},\tau)$ with $\tau\in G$. By abuse of notation, we use $Z_2$ to denote the system as well the underlying nilmanifold $G/\Gamma$.


	In view of ergodicity, the topological system $(Z_2, \tau)$ is minimal (see, for instance, \cite[Theorem 4.1.1]{Bergelson_Host_Kra05}).
	We can assume that $G$ is generated by the connected component of the identity and $\tau$.
	Indeed, the projection of the connected component of $G$ onto $Z_2$ is an open subset of $Z_2$ (as its pre-image under the natural map $G\to Z_2$ is the union of all connected components of $G$ having non-empty intersection with $\Gamma$ and hence it is open), and by minimality, its orbit under $\tau$ is all of $Z_2$. Therefore, if we let $\tilde G$ be the subgroup of $G$ generated by the connected component of the identity and $\tau$, it follows that $Z_2=\tilde G/(\Gamma\cap\tilde G)$.
	
	Since $G$ is a $2$-step nilpotent group, the commutator $G_{2}=[G,G]$ is inside the center of $G$, and hence the subgroup $\Gamma_{2}=G_{2}\cap \Gamma$ is normal in $G$.
	Therefore $Z_2=(G/\Gamma_{2})/(\Gamma/\Gamma_{2})$ and thus after modding out by $\Gamma_2$ we can assume that $G_2\cap\Gamma=\{1\}$, which implies that $G_{2}$ is a compact abelian Lie group.
	From \cite[Theorem 4.1.4]{Bergelson_Host_Kra05}, it follows that $G_{2}$ is connected, and so $G_{2}$ must be a finite dimensional torus.
	
	Let $Z_1:= G/(G_2 \Gamma)$ be the Kronecker factor of $Z_2$ and note that it is also a compact abelian Lie group. For $\delta > 0$, define $S_{\delta}$ as in \eqref{equ:s}.
	It now suffices to show that
	\begin{equation}\label{equ:ZLF1}
	\begin{split}
	\lim_{\delta\to 0}\lim_{N-M\to\infty}\frac{1}{\vert S_{\delta}\cap[M,N)\vert}\sum_{n\in S_{\delta}\cap[M,N)}\int_{X}\prod_{i=1}^4f(T^{a_{i}n}x)\, d\mu(x)\geq C\left(1-\frac4{\ell}\right)^{\ell}\left(\int_{X} f\,d\mu\right)^{\ell}
	\end{split}
	\end{equation}
	for all $0\leq f\leq 1$.
	By \cref{proposition:characteristic-factor-along-bohr-set}, the left hand side of (\ref{equ:ZLF1}) is 0 if we replace at least one of the four $f$'s with $f-\mathbb{E}(f\vert{Z_{2}})$. 	
	Hence it suffices to prove (\ref{equ:ZLF1}) under the assumption that $X=Z_2=G/\Gamma$.
	
	Since $G$ is 2 step nilpotent,
	by \cref{prop:sz}, the left hand side of (\ref{equ:ZLF1}) equals to
	\begin{equation}\label{123}
	\int_{Z_2}\int_{G_2}\int_{G_2} \prod_{i=1}^4f\left(g g_1^{a_i \choose 1} g_2^{a_i \choose 2} \Gamma\right)\, d\mu_{G_{2}}(g_2)\, d\mu_{G_{2}}(g_1)\,d\mu_{Z_2}(g\Gamma),
	\end{equation}
	where $\mu_{X}$ and $\mu_{G_{2}}$ are the Haar measures on $X$ and $G_{2}$ respectively.
	Recall that $G_{2}$ is a torus, say $G_2=\mathbb{T}^{m}$.
	Consider the subgroup
	$$Y:=\left\{(y_{1},\dots,y_{4})\in(\mathbb{T}^m)^4\colon(\exists g_1,g_2\in\T^m)\ y_i=\binom{a_i}1g_1+\binom{a_i}2g_2\right\}\subset\T^{4m},$$
	where we now use the additive notation.
	Then we may rewrite
	$$\eqref{123}=\int_{Z_2}\int_{Y}\prod_{i=1}^4f(gy_{i}\Gamma)\,d\mu_{Y}(y_{1}y_2,y_3,y_{4})\,d\mu_{Z_2}(g\Gamma),$$
	where $\mu_Y$ is the Haar measure on $Y$.
	We can also describe $Y$ in terms of $V$ as $Y=\big(\overline{V}/\mathbb{Z}^{4}\big)^{m}$, where $\overline{V}$ is the closure $V$ in $\R^4$ (or, equivalently, its $\mathbb{R}$-span).

	For each $g\in G$, let $f_{g}\colon G_{2}\to\mathbb{R}$ be the function defined by the formula $f_{g}(g_{2}\Gamma)=f(gg_{2}\Gamma)$ for all $g_{2}\in G_{2}$. 	
	Then by \cref{lem:equ}, $(1)\Rightarrow (4)$, and then Jensen's inequality, we conclude that
	\begin{equation}\nonumber
	\begin{split}
	&\quad\int_{Z_2}\int_{Y}\prod_{i=1}^4f(gy_{i}\Gamma)\,d\mu_{Y}(y_{1}y_2,y_3,y_{4})\,d\mu_{Z_2}(g\Gamma)
	\geq C\left(1-\frac{4}{\ell}\right)^{\ell}\int_{Z}\left(\int_{G_{2}}f_{g}\,d\mu_{G_{2}}\right)^{\ell}\,d\mu_{Z}(g\Gamma)
	\\&\geq C\left(1-\frac{4}{\ell}\right)^{\ell}\left(\int_{Z_2}\int_{G_{2}}f_{g}\,d\mu_{G_{2}}\,d\mu_{Z_2}(g\Gamma)\right)^{\ell}
	=C\left(1-\frac{4}{\ell}\right)^{\ell}\left(\int_{Z_2}f\,d\mu_{Z_2}\right)^{\ell}.
	\end{split}
	\end{equation}

\end{proof}

\subsection{Optimal recurrence implies abundance of solutions for \texorpdfstring{$\mathbf{k = 4}$}{k=4}}
\label{sec_proofintegerexample}

In this subsection we prove Theorem \ref{thm:integerinverse}.
We need the following well known equidistribution result whose short proof we include for completeness.

\begin{lemma}\label{lem:eqd1}
	For every Bohr$_{0}$ set $S$, every $\a\in\mathbb{R}^{m}$ whose coordinates are rationally independent, and every cube $I\subseteq\mathbb{T}^{m}$,
	\[
	    \lim_{N-M\to\infty}\frac{\big\vert\big\{n\in S\cap [M,N)\colon n^{2}\a\bmod\mathbb{Z}^{m}\in I\big\}\big\vert}{\vert S\cap [M,N)\vert}=\mu_{\mathbb{T}^m}(I).
	\]
\end{lemma}

\begin{proof}
	By assumption, we can write $S=\{n\in\mathbb{N}\colon nx\in U\}$, where $K$ is a compact abelian group, $U\subseteq K$ is a neighborhood of $1_{K}$ such that $\one_{U}$ is Riemann integrable, and $x\in U$ is a point such that $\overline{\{nx\colon n\in\Z\}}=K$ and $S=\{n\in\mathbb{N}\colon nx\in U\}$.
	Then it suffices to show that as $N-M\to\infty$,
	
	\[	\frac{\vert\{n\in S\cap [M,N)\colon n^{2}\a\in I\}\vert}{N-M} =	\frac{\vert\{n\in [M,N)\colon  (nx,n^2\a)\in U\times I\}\vert}{N-M}   \]
	converges to $\mu_K(U)\times \mu_{\mathbb{T}^{m}}(I)=\mu_{K\times\mathbb{T}^{m}}(U\times I)$, where $\mu_K$, $\mu_{\mathbb{T}^{m}}$ and $\mu_{K\times\mathbb{T}^{m}}$ are the Haar measures on $K$, $\mathbb{T}^{m}$ and $K\times\mathbb{T}^{m}$, respectively. This follows once we show that the sequence $(nx,n^2\a)_{n\in\N}$ is well distributed on $K\times\T^{m}$.
	Since $\one_{U}$ is Riemann integrable, it suffices to show that for every character $\chi\colon K\to S^1\subset\C$ of $K$ and every $\mathbf{b}\in\Z^{m}$, if either $\chi$ is non-trivial or $\mathbf{b}\neq\mathbf{0}$, then
	\begin{equation}\label{eq_welldistributedlemma}
	\lim_{N-M\to\infty}\frac1{N-M}\sum_{n=M}^{N-1}e^{2\pi i\mathbf{b}\cdot\a n^{2}}\chi(x)^n=0.
	\end{equation}
	Let $\theta\in\T$ be such that $\chi(x)=e^{2\pi i\theta}$.
	Since the group generated by $x$ is dense in $K$, $\theta\notin\Z$ unless $\chi$ is trivial.
	Then we can write $e^{2\pi i\mathbf{b}\cdot\a n^{2}}\chi(x)^n=e^{2\pi i(n^2\mathbf{b}\cdot\a+n\theta)}$.
	If $\mathbf{b}\neq\mathbf{0}$ then \eqref{eq_welldistributedlemma} follows from Weyl's equidistribution theorem, and if $\mathbf{b}=\mathbf{0}$ then $\theta\notin\Z$ and \eqref{eq_welldistributedlemma} follows quickly from evaluating the resulting geometric series.
\end{proof}

\begin{proof}[Proof of Theorem \ref{thm:integerinverse}]

	Let $\ell,a_{1},\dots,a_{4},C$ and $V$ be as in the statement of the theorem.
Since $V$ has a basis of rational vectors, there exists a positive constant $\epsilon$ such that any point $\x\in\Z^4$ at a distance (say in the $\ell^\infty$ norm) less than $\epsilon$ from the closure $\overline{V}\subset\R^4$ must in fact belong to $V$.
Fix $m,N_0\in\N$ and $E\subset[N_0]^m$.

	Let $X=\mathbb{T}^{2m}$ be the $2m$ dimensional torus endowed with the Lebesgue measure $\mu$. Define the map $T\colon X\to X$ via the formula
	$T(\x,\y)=(\x+\a,\y+2\x+\a), \x,\y\in\mathbb{T}^{m}$ for some $\a\in\mathbb{T}^{m}$ whose coordinates are rationally independent.
	Then $(X,\mathcal{B},\mu,T)$ is ergodic, and for each $n\in\N$,
	\begin{equation}\nonumber
	\begin{split}
	T^{n}(\x,\y)=(\x+n\a,\y+2n\x+n^{2}\a).
	\end{split}
	\end{equation}
	For $\i=(c_{1},\dots,c_{m})\in[N_0]^{m}$, denote
	$$B_{\i}= \left[\frac{c_{1}}{N_0},\frac{c_{1}+\epsilon}{N_0}\right) \times\dots\times\left[\frac{c_{m}}{N_0},\frac{c_{m}+\epsilon}{N_0}\right),\qquad A=\bigcup_{\i\in E}\T^{m}\times B_{\i}.$$
	We can directly compute that $\mu(A)=\frac\epsilon{N_0^{m}}|E|=\epsilon^md_{m,N_0}(E)$.
	Let $(\x,\y)\in\T^{2m}$ and $n\in\N$.
	Note that for all $1\leq j\leq 4$, $T^{a_{j}n}(\x,\y)\in A$ if and only if
	\begin{equation}\label{equ:ex3}
	\begin{split}
	u_{j}=u_j(\x,\y,n):=\y+2a_{j}n\x+a^{2}_{j}n^{2}\a\in B_{\i_{j}}
	\end{split}
	\end{equation}
	for some $\i_{j}\in E$.
	Fix such a point $(\x,\y,n)\in\T^{2m}\times\N$. Then the vector
	\begin{equation}\nonumber
	\begin{split}
	(u_{1},\dots,u_{4})=\y(1,\dots,1)+2n\x(a_{1},\dots,a_{4})+ n^{2}\a(a_{1}^{2},\dots,a_{4}^{2})\in B_{\i_{1}}\times\dots\times B_{\i_{4}}
	\end{split}
	\end{equation}
	belongs to the closure in $\R^{m\times 4}$ of $V^{m}$.
	Since $N_0u_{j}$ is at most $\epsilon$ away from the integer vector $\i_{j}$ (in the $\ell^\infty([m])$ distance), from the definition of $\epsilon$ we deduce that $(\i_{1},\dots,\i_{4})$ belongs to $V^{m}$ as well.
	Let $W$ denote the collection of all tuples $(\i_{1},\dots,\i_{4})\in V^{m}$ with $\i_{j}\in E$.
	By definition, $D_{m,N_0}(V^{m},E)=\frac{\vert W\vert}{\vert V\cap[N_0]^{4}\vert^{m}}$.
	
	By the discussion above,
	\begin{equation}\label{equ:ex4}
	\mu(T^{a_{1}n}A\cap\dots\cap T^{a_{4}n}A)
	=
	\sum_{(\i_{1},\dots,\i_{4})\in W} \int_{X}\prod_{j=1}^{4}\one_{B_{\i_{j}}}\big(u_{j}(\x,\y,n)\big)\,d\mu(\x,\y).
	\end{equation}

	Fix $(\i_{1},\dots,\i_{4})\in W$.
If $n\in\N$ is such that (\ref{equ:ex3}) holds for all $1\leq i\leq 4$, then considering (\ref{equ:ex3}) as a linear equation system with $4m$ equations and the coordinates of $\y,n\x,n^{2}\a$ as unknowns (i.e. $3m$ unknowns in total)
	, we deduce that there exist 3 cubes $I_{1}, I_{2}$ and $I_{3}$ in $\T^{m}$ with side length at most $\frac{1}{N_0}$ such that $\y\in I_{1},n\x\in I_{2},n^{2}\a\in I_{3}$.

	\cref{lem:eqd1} implies that for any Bohr$_{0}$ set $S$,
	$$\lim_{N-M\to\infty}\frac{\vert\{n\in S\cap [M,N)\colon n^{2}\a\in I_{3}\}\vert}{\vert S\cap [M,N)\vert}=\mu(I_{3})\leq\frac{1}{N_0^{m}}.$$
	Since $\mu(I_{1})\leq \frac{1}{N_0^{m}}$ and $\mu(\{\x\in\mathbb{T}^{m}\colon n\x\in I_{2}\})=\mu(I_{2})\leq \frac{1}{N_0^{m}}$ for any $n\neq 0$, by \eqref{equ:ex4} we conclude that
	\begin{equation}\nonumber
	\begin{split}
	&\quad D_{m,N_{0}}(V^{m},E)=\sum_{(\i_{1},\dots,\i_{4})\in W}\frac{1}{\vert V\cap [N_{0}]^{4}\vert^{m}}
	\\&\geq \frac{N^{3m}_{0}}{\vert V\cap [N_{0}]^{4}\vert^{m}}\sum_{(\i_{1},\dots,\i_{4})\in W}\limsup_{N-M\to\infty}\frac{1}{\vert S\cap [M,N)\vert}\sum_{n\in S\cap [M,N)}\int_{X}\prod_{j=1}^{4}\one_{B_{\i_{j}}}(u_{j}(\x,\y,n))\,d\mu(\x,\y)
	\\&\geq \frac{N^{3m}_{0}}{\vert V\cap [N_{0}]^{4}\vert^{m}}\limsup_{N-M\to\infty}\frac{1}{\vert S\cap [M,N)\vert}\sum_{n\in S\cap [M,N)}\mu(T^{a_{1}n}A\cap\dots\cap T^{a_{d}n}A)
	\\&\geq C\frac{N^{3m}_{0}}{\vert V\cap [N_{0}]^{4}\vert^{m}}\mu(A)^{\ell}=C\frac{\epsilon^{m\ell}N^{3m}_{0}}{\vert V\cap [N_{0}]^{4}\vert^{m}}d_{m,N_0}(E)^\ell,
	\end{split}
	\end{equation}
	which finishes the proof by taking $\beta<\lim\limits_{N\to\infty}\frac{\epsilon^{\ell} N_{0}^3}{\vert V\cap [N_{0}]^{4}\vert}$.
\end{proof}

\section{Optimal recurrence along families of polynomials}\label{s5}

\subsection{Proof of \texorpdfstring{\cref{proposition:globn2connected}}{Proposition 1.12}}
In this section, we prove a generalization of \cref{proposition:globn2connected}.
In order to state the result, we need to introduce a notion defined and studied by Leibman in \cite{Leibman10b}. The \emph{$C$-complexity} ($C$ stands for \emph{connected}) of a family of integer-valued polynomials $\{p_1,\ldots,p_d\}$ is the minimum $k \in \Z$ for which the factor $Z_k$ is characteristic for this family in an ergodic nilsystem $(G/\Gamma,\mathcal{B},\mu,T)$ with $G$ connected.

Note that this notion is different from the standard complexity as the latter is defined for general ergodic systems instead of nilsystems with connected Lie groups.

Let $P = \{p_1, p_2, p_3\}$ be a family of $3$ integer-valued polynomials. Following \cite[Section 8]{Leibman10b}, the $C$-complexity of $P$ is $1$ if
\begin{itemize}
    \item $p_1, p_2, p_3$ are linearly independent over $\Q$ (the standard complexity is also $1$ in this case (see \cite{Frantzikinakis_Kra05b})),
    \item or there exist two linearly independent over $\Q$ polynomials  $q_1, q_2$ such that $p_1 = a q_1$, $p_2 = b q_2$ and $p_3 = c q_1 + d q_2$ for some $a, b, c, d \in \Z$.
\end{itemize}
For the other remaining case where $p_1 = a q, p_2 = b q$ and $p_3 = c q$ for some non-constant polynomial $q$ and $a, b, c$ distinct non-zero, the $C$-complexity of $P$ is $2$.

\begin{lemma} \label{proposition:Bohr-Poly-Average}
	\label{proposition:characteristic-factor-along-bohr-set-poly}
	Let $(G/\Gamma, \mathcal{B}, \mu, T)$ be an ergodic nilsystem where $G$ is connected.
	Let $Z_1$ be its Kronecker factor and $\alpha\in Z_1$ be the rotation induced by $T$.
	Suppose $q_1$ and $q_2$ are two integer-valued polynomials that are linearly independent over $\mathbb{Q}$ and both have $0$ constant terms. Let $p_1=aq_1$, $p_2=bq_2$ and $p_3=cq_1+dq_2$ for $a,b,c,d \in \mathbb{Z}$.
	
	For $\delta>0$, let $B_{\delta}$ be the ball in $Z_1$ centered at $0$ of radius $\delta$ and define $S_{\delta} = \{n\in \mathbb{N}: (q_1(n)\alpha,q_2(n)\alpha)\in B(\delta)\times B(\delta)\}$. Let $f_1, f_2, f_3 \in L^{\infty}(\mu)$ and assume $\mathbb{E}(f_i | Z_1) = 0$ for some $1 \leq i \leq 3$. Then
	\begin{equation}
	\label{equation:char-bohr-poly}
	\lim_{N - M \to \infty} \frac{1}{|S_{\delta} \cap [M, N)|} \sum_{n \in S_{\delta} \cap [M,N)} f_1(T^{p_1(n)} x) f_2(T^{p_2(n)}x ) f_3(T^{p_3(n)}x) = 0,
	\end{equation}
	where the limit is taken in $L^{2}(\mu)$.
\end{lemma}

\begin{proof} The proof is similar to that of \cref{proposition:characteristic-factor-along-bohr-set}, subjected to some minor changes that we write explicitly. In this proof, all the limits are taken in $L^{2}(\mu)$.
	Without loss of generality, we assume $\mathbb{E}(f_1|Z_1) = 0$.
	Let $L$ be the limit on the left hand side of (\ref{equation:char-bohr-poly}) and $d(S_{\delta})$ be the uniform density of $S_{\delta}$. Since $\{(q_1(n)\alpha,q_2(n)\alpha)\}$ is well distributed on $Z_1\times Z_1$ (because $q_1$ and $q_2$ are linearly independent over $\mathbb{Q}$), we have
	\begin{multline}
	d(S_{\delta}) L = \lim_{N - M \to \infty} \frac{1}{N-M} \sum_{n=M}^{N-1} \mathbbm{1}_{S_{\delta}}(n) f_1(T^{p_1(n)}x) f_2(T^{p_2(n)} x) f_3(T^{p_3(n)}x) = \\
	\lim_{N - M \to \infty} \frac{1}{N-M}\sum_{n=M}^{N-1} \mathbbm{1}_{B(\delta)\times B(\delta)}(q_1(n) \alpha,q_2(n)\alpha) f_1(T^{p_1(n)}x) f_2(T^{p_2(n)} x) f_3(T^{p_3(n)}x).
	\end{multline}
	Approximating the Riemann integrable function $\mathbbm{1}_{B(\delta)\times B(\delta)}$ by finite linear combination of characters, it suffices to show
	\begin{equation} \label{equation:char-bohr-poly-1.5}
	\lim_{N - M \to \infty} \frac{1}{N-M}\sum_{n=M}^{N-1} \chi_1 (q_1(n) \alpha)\chi_2(q_2(n)\alpha) f_1(T^{p_1(n)}x) f_2(T^{p_2(n)} x) f_3(T^{p_3(n)}x) = 0
	\end{equation}
	for all characters $(\chi_1,\chi_2)$ of $Z_1\times Z_1$. Note that the limit in the left hand side of (\ref{equation:char-bohr-poly-1.5})
	is equal to
	\begin{equation} \label{equation:char-poly}
	\bar{\chi}_1(y)\bar{\chi}_2(z) \lim_{N - M \to \infty} \frac{1}{N-M}\sum_{n=M}^{N-1} \chi_1(q_1(n) \alpha + y)\chi_2(q_2(n) \alpha + z) f_1(T^{p_1(n)}x) f_2(T^{ p_2(n)} x) f_3(T^{p_3(n)}x)
	\end{equation}
	for every $y,z\in Z_1$.
	Since $G$ is connected, there exist $g,h \in G$ such that $a g = \alpha$ and $bh=\alpha$. Let $\alpha/a$ and $\alpha/b$ denote the elements $g$ and $h$, respectively. Consider the system $Y = (X \times Z_1\times Z_1, \mathcal{B} \times \mathcal{G}\times \mathcal{G}, \mu \times \mu_{Z_1}\times \mu_{Z_1}, \tilde{T})$, where $\tilde{T}=T \times (\alpha/a)\times (\alpha/b)$.  We can write then \begin{multline}
	\bar{\chi}_1(y)\bar{\chi}_2(z) \lim_{N - M \to \infty} \frac{1}{N-M}\sum_{n=M}^{N-1} \chi_1(q_1(n) \alpha + y)\chi_2(q_2(n) \alpha + z) f_1(T^{p_1(n)}x) f_2(T^{ p_2(n)} x) f_3(T^{p_3(n)}x) \\
	= \lim_{N - M \to \infty}\frac{1}{N-M}\sum_{n=M}^{N-1} \tilde{T}^{p_1(n)}f_1 \otimes 1 \otimes 1 \cdot \tilde{T}^{p_2(n)} f_2 \otimes \chi_1 \otimes 1\cdot  \tilde{T}^{p_3(n)} f_3 \otimes 1\otimes \chi_2.
	\end{multline}

	Since $\mathbb{E}\big(f_1|Z_1(X)\big) = 0$, for almost every ergodic component $Y_t$ of $Y$, we have $\mathbb{E}(f_1 \otimes \chi_1\otimes 1|Z_1(Y_t)) = 0$ (one way to verify this is to show $\lVert f_1 \otimes \chi_1 \otimes 1 \rVert_{2} = 0$, where $\lVert \cdot \rVert_k$ is the Host-Kra's seminorm defined in \cite{Host_Kra05}).

		Also, almost every ergodic component $Y_t$ can be written as $G_{t}/\Gamma_t$ with $G_{t}$ being connected. To see this, let $\widehat{Y}_t$ be the closure of the $\R$-flow
	associated to $\tilde{T}$, i.e. $(t\cdot \tilde{T})_{t\in \R}$ of a point in $Y_t$. By a special case of Ratner's theorems \cite{Ratner91a} we have that $\widehat{Y}_t=G_{t}/\Gamma_t$ where $G_t$ is a closed connected subgroup of $G$ and $\Gamma_t=G_t\cap \Gamma$.
	We claim that $\widehat{Y}_t={Y}_t$ and for this, it suffices to show that the orbit under $\tilde{T}$ of a point in ${Y}_t$ is dense in  $\widehat{Y}_t$. To prove this, it suffices to show that the projection of this orbit onto $G_t/([G_t,G_t]\Gamma_t)$ is dense. This latter space is a factor of a translation of the closure of  $\{ (t\alpha, t\alpha/a, t\alpha/b) : t\in \R\}$ (in $Z_1\times Z_1\times Z_1$), while the orbit of a point under $\tilde{T}$ is a translation of the closure of $ \{ (n\alpha, n\alpha/a, n\alpha/b) :n\in \Z \}$. Since the closures of $\{ (t\alpha, t\alpha/a, t\alpha/b) : t\in \R\}$ and $\{(n\alpha, n\alpha/a, n\alpha/b) :n\in \Z \}$ coincide we get the desired conclusion.

 By assumption, the $C$-complexity of  $\{p_1,p_2,p_3\}$ is $1$, and we get  \begin{equation}
	\lim_{N - M \to \infty}\frac{1}{N-M}\sum_{n=M}^{N-1} \tilde{T}^{p_1(n)}f_1 \otimes 1 \otimes 1 \cdot \tilde{T}^{p_2(n)} f_2 \otimes \chi_1 \otimes 1\cdot  \tilde{T}^{p_3(n)} f_3 \otimes 1\otimes \chi_2 =0
	\end{equation}
	for almost every $t$.  It follows that \eqref{equation:char-poly} equals to 0 in $L^2(\mu\times \mu_{Z_1}\times \mu_{Z_1})$, which implies that the left hand side of \eqref{equation:char-bohr-poly} is equal to $0$ in $L^2(\mu)$.  This finishes the proof.
\end{proof}
\medbreak

\begin{lemma} \label{proposition:globn2nilconnected}
	Let $(G/\Gamma,\mu,T)$ be a nilsystem with $G$ connected. Let $p_1,p_2,p_3$ be three polynomials as in \cref{proposition:Bohr-Poly-Average}.
	Then
	for every $A\in\mathcal{B}$ and every $\epsilon>0$, the set
	\[\big\{n\in\N\colon\mu(A\cap T^{-p_1(n)}A \cap T^{-p_2(n)}A \cap T^{-p_3(n)}A)>\mu(A)^{4}-\epsilon\big\}\]
	is syndetic.
\end{lemma}

\begin{proof}
	Let $\epsilon>0$, $A\in \mathcal{B}$ and set $f=\mathbb{E}(\one_A\vert Z_1)$. Let  $\delta'>0$ be such that the translation $f_t(\cdot)=f(\cdot+t)$ satisfies that $\| f -f_t\|_{L^1(\mu_{Z_1})} < \frac{\epsilon}{3}$ if $t \in B(\delta')$.
	Let $\delta '>\delta>0$ be such that if $q_1(n)\alpha$ and $q_2(n)\alpha$ are in $B(\delta)$, then $p_1(n)\alpha,p_2(n)\alpha$ and $p_3(n)\alpha$ are in  $B(\delta')$.
	Then, for $n \in S_{\delta}$, we have that $\|f-T^{p_i(n)}f\|_{L^1(\mu_{Z_1})}<\frac{\epsilon}{3}$ for $i=1,2,3$ and thus
	
	\begin{equation} \label{equation:good-poly-kronecker}\int f\cdot T^{p_1(n)} f\cdot T^{p_2(n)} f\cdot T^{p_3(n)} f ~ d\mu_{Z_1} > \int f^4d\mu_{Z_1} - 3\frac{\epsilon}{3} \geq \mu(A)^4-\epsilon.
	\end{equation}
	
	By \eqref{equation:good-poly-kronecker} and Proposition \ref{proposition:Bohr-Poly-Average}, we get that
	\begin{equation}
	\lim_{N - M \to \infty} \frac{1}{|S_{\delta} \cap [M, N)|} \sum_{n \in S_{\delta} \cap [M,N)} \mu( A\cap T^{-p_1(n)}(A) \cap T^{-p_2(n)}A \cap T^{-p_3(n)}A) \geq  \mu(A)^4-\epsilon,
	\end{equation}  	 	  	
which finishes the proof.
\end{proof}

\begin{proposition}
	\label{proposition:globn2connected-general}
	Let $p_1, p_2, p_3$ be as in \cref{proposition:characteristic-factor-along-bohr-set-poly}.
	Let $(X,\mathcal{B},\mu,T)$ be an ergodic system and $Z_3$ the 3-step nilfactor of $X$. Assume $Z_3$ is the inverse limit of nilsystems that can be represented as $G/\Gamma$ with $G$ connected. Then for every $A \in \mathcal{B}$ and $\epsilon > 0$, the set
	\begin{equation}\nonumber
	\big\{n\in\N\colon\mu(A\cap T^{-p_1(n)}A \cap T^{-p_2(n)}A \cap T^{-p_3(n)}A)>\mu(A)^{4}-\epsilon\big\}
	\end{equation}
	is syndetic.
\end{proposition}

\begin{proof}
	Define $a(n)=\int \one_A \cdot \one_A \circ T^{p_1(n)} \cdot \one_A \circ T^{p_2(n)} \cdot \one_A \circ T^{p_3(n)} d\mu$ and $\tilde{a}(n)=\int \mathbb{E}( \one_A \vert Z_3) \cdot \mathbb{E}( \one_A \vert Z_3)\circ T^{p_1(n)} \cdot  \mathbb{E}( \one_A \vert Z_3) \circ T^{p_2(n)} \cdot  \mathbb{E}( \one_A \vert Z_3)\circ T^{p_3(n)} d\mu$.  We claim that  \[\limsup\limits_{N-M\to \infty} \frac{1}{N-M}\sum_{n=M}^{N-1} |a(n)-\tilde{a}(n)|^2= 0.\]
	The proof is essentially given in \cite[Corollary 4.5]{Bergelson_Host_Kra05}. Using a telescoping difference between $a(n)$ and $\tilde{a}(n)$, it suffices to show that
	if some $f_i, 0\leq i\leq 3$ has 0 conditional with respect to $Z_3(X)$, then
	\[ \lim_{N-M\to\infty}\frac{1}{N-M} \sum_{n=M}^{N} \left (\int f_0(x)f_1(T^{p_1(n)}x)f_2(T^{p_2(n)}x)f_3(T^{p_3(n)}x)d\mu \right)^2=0. \]

	We assume without loss of generality that  $\mathbb{E}(f_{0}\vert Z_3(X))=0$. Let $\mu \times \mu=\int_{Z} d\mu_s dm(s)$ be the ergodic decomposition of $\mu\times \mu$ under $T\times T$.
	By \cite[Proposition 4.3]{Bergelson_Host_Kra05}, for almost every $s$, $\mathbb{E}(f_{0}\otimes f_{0} \vert Z_2 (X\times X))=0$, where $X \times X$ is endowed with the measure $\mu_s$ and the transformation $T\times T$.	
		By \cite[Theorem B]{Frantzikinakis08}, the 2-step nilfactor is characteristic for the average $\lim \limits_{N-M\to\infty}\frac{1}{N-M}\sum_{i=M}^{N-1} T^{p_1(n)} f_1 \cdot T^{p_2(n)} f_2 \cdot T^{p_3(n)} f_3$,  for any bounded measurable functions $f_1,f_2,f_3$ of any measure preserving system. Therefore, the limit as $N-M$ goes to infinity of{\small
		\begin{equation}  \label{equation:cerolimituniform} \frac{1}{N-M} \sum_{n=M}^{N} \int f_0(x)f_0(x')f_1(T^{p_1(n)}x)f_1(T^{p_1(n)}x')f_2(T^{p_2(n)}x)f_2(T^{p_2(n)}x')f_3(T^{p_3(n)}x)f_2(T^{p_3(n)}x')d\mu_s(x,x') \end{equation}}
		is equal to 0 for almost every $s$. Integrating \eqref{equation:cerolimituniform} with respect to $s$, we deduce the claim.

	By the claim, it suffices to prove the result under the assumption that $X=Z_3$. By an approximation argument, we can assume that $(X=G/\Gamma,\mathcal{B},\mu,T)$  where $G$ is connected. Proposition \ref{proposition:globn2nilconnected} give us the desired conclusion.
\end{proof}

\begin{proof}[Proof of \cref{proposition:globn2connected}]
It follows immediately from \cref{proposition:globn2connected-general}, with $p_1(n) = n, p_2(n) = n^2$ and $p_3(n) = 2n$ for $n \in \Z$. 	
\end{proof}

\subsection{Proof of \texorpdfstring{\cref{proposition:commutingpoly}}{Proposition 1.13}}
Let $X=\mathbb{T}^2$, $\mu$ be the Lebesgue measure on $\mathbb{T}^2$,  $T_1\colon (x,y) \mapsto (x+\alpha,y+2x+\alpha)$, and $T_2\colon (x,y)\mapsto (x,y-2\alpha)$. Then $T_1$ and $T_2$ commute, preserve the measure  $\mu$, and moreover, $T_1$ is ergodic for $\mu$.
	
	We have that $T_1^n(x,y)=(x+n\alpha, y +2nx+n^2\alpha)$, $T_1^{2n}(x,y)=(x+2n\alpha,y+4nx+4n^2\alpha)$ and $T_2^{n^2}(x,y)=(x,y-2n^2\alpha)$. Write $u=y +2nx+n^2\alpha$, $v=y+4nx+4n^2\alpha$ and $w=y+2n^2\alpha$.
	Then $v-2u+w=0$. 	
	Let $\Lambda\subseteq [N]$ be a subset with no arithmetic progression of length 3 and  set $A=\mathbb{T}\times \bigcup\limits_{a\in \Lambda} (\frac{a}{N}-\frac{1}{4N},\frac{a}{N}+\frac{1}{4N})$.
	
	If $\one_A(x,y)\one_A(T_1^n(x,y))\one_A(T_1^{2n}(x,y)) \one_A(T_2^{n^2}(x,y))>0$, then there exist $a_0,a_1,a_2\in \Lambda$ such that $ u \in (a_0-\frac{1}{4n},a_0+\frac{1}{4N})$, $v\in (a_1-\frac{1}{4N},a_1+\frac{1}{4N}) $, $w\in (a_2-\frac{1}{4N},a_2+\frac{1}{4N})$. Thus
	\[ a_1-2a_0+a_2  + t = 0 \]	
	for some $|t|\leq \frac{1}{N}$. So $a_1-2a_0+a_2=0$ and  then $a_0=a_1=a_2$. It follows that $nx\in (1/4N,1/4N)$ and
	
	\[ \int \one_A(x,y)\one_A(T_1^n(x,y))\one_A(T_1^{2n}(x,y) \one_A(T_2^{n^2}(x,y)) d\mu(x,y) \leq  \frac{1}{N^2}|\Lambda|. \]
	A quick computation shows that
	\[ \frac{1}{N^2}|\Lambda| \leq \left (\frac{|\Lambda|}{2N}\right )^{\ell}=\mu(A)^{\ell}   \]
	as long as $\ell\leq \frac{2\log(N)-\log(|\Lambda|) }{\log(2N)-\log(|\Lambda|)}.$
	In view of \cite{Behrend46}, we can take $\Lambda$ of cardinality $N^{1-\epsilon}$ for any $\epsilon>0$, and hence we can make the fraction $\frac{2\log(N)-\log(|\Lambda|) }{\log(2N)-\log(|\Lambda|)}$ arbitrarily large, finishing the proof.
\qed

\bibliographystyle{plain}
\bibliography{refs}

\end{document}